\documentclass[a4paper,12pt,reqno]{NumPDEsArticle}
\usepackage[T1]{fontenc}
\usepackage[utf8]{inputenc}
\usepackage[english]{babel}
\usepackage{amsmath, amsthm, amssymb}
\usepackage{siunitx}
\usepackage{fullpage}
\usepackage{xcolor}
\usepackage{etoolbox}
\usepackage{empheq}
\usepackage{enumitem}
\usepackage{caption}
\usepackage{subcaption}
\usepackage{chemformula}
\usepackage{pgfplotstable}
\usepackage{xparse}
\usepackage{mathtools}
\usepackage{placeins}
\usepackage{bm}
\pgfplotsset{
compat = newest,
tick label style = {font = \tiny},
legend style = {font = \tiny},
xlabel style={yshift=+0.5ex},
ylabel style={yshift=-1.0ex}
}
\usepackage[width=0.95\textwidth]{caption}
\usepackage[width=0.95\textwidth,labelfont=rm]{subcaption}
\usepackage{doi, hyperref}
\makeatletter
\patchcmd{\@xfloat}{\def\@fps{#2}}{\def\@fps{#2p}}{}%
  {\PackageError{numpdes-layout}{Could not enable float pages}{}}
\makeatother
\numberwithin{theorem}{section}

\BeforeBeginEnvironment{abstract}{\maketitle}

\newcommand\aaa{\boldsymbol{a}}
\newcommand\bb{\boldsymbol{b}}
\newcommand\ee{\boldsymbol{e}}

\newcommand\mm{\boldsymbol{m}}
\newcommand\nn{\boldsymbol{n}}
\newcommand\uu{\boldsymbol{u}}

\newcommand\vv{\boldsymbol{v}}

\newcommand\xx{\boldsymbol{x}}

\newcommand\CC{\boldsymbol{C}}

\newcommand\ff{\boldsymbol{f}}
\newcommand\E{\mathcal{E}}

\newcommand\MM{\bm{M}}

\newcommand\VV{\bm{V}}
\newcommand\II{\mathcal{I}}
\newcommand\NN{\mathcal{N}}

\renewcommand\S{\mathcal{S}}
\newcommand\T{\mathcal{T}}
\newcommand{\heff}{\boldsymbol{h}_{\rm eff}}
\newcommand{\zz}{\boldsymbol{z}}

\newcommand{\pphi}{\boldsymbol{\phi}}
\newcommand{\vvarphi}{\boldsymbol{\varphi}}

\newcommand{\bT}{\boldsymbol{T}}

\newcommand{\bH}{\mathbf{H}}
\newcommand{\bL}{\mathbf{L}}

\newcommand{\bW}{\mathbf{W}}

\newcommand{\OT}{\Omega_T}

\newcommand{\mhat}{\widehat{\mm}}
\newcommand\ppsi{\boldsymbol{\psi}}
\newcommand\zzeta{\boldsymbol{\zeta}}

\newcommand\0{\boldsymbol{0}}

\newcommand\Grad{\boldsymbol{\nabla}}
\newcommand\Lapl{\boldsymbol{\Delta}}

\newcommand\R{\mathbb{R}}
\newcommand\dt{\mathrm{d}_t}
\newcommand\dtt{\,\mathrm{d}t}

\newcommand\N{\mathbb{N}}

\newcommand{\abs}[1]{\left\lvert #1 \right\rvert}
\newcommand{\dual}[3][]{\langle #2,#3 \rangle_{#1}}
\NewDocumentCommand{\inner}{O{} O{} m m}{%
  #2\langle #3, #4 #2\rangle%
  \IfValueT{#1}{_{#1}}%
}
\newcommand{\norm}[2][]{\left\lVert #2 \right\rVert_{#1}}
\DeclareMathOperator{\diam}{diam}

\newcommand\mmt{\partial_t \mm}
\newcommand\mmtt{\partial_{tt} \mm}

\def\addlegendimage{\csname pgfplots@addlegendimage\endcsname}

\begin{document}
\title{An unconditionally convergent BDF2-type \\finite element integrator for the \\inertial Landau--Lifshitz--Gilbert equation}
\author{Michele~Ald\'e}
\address{TU Wien, Institute of Analysis and Scientific Computing
(E 101), Wiedner Hauptstraße 8--10,
1040 Vienna, Austria}
\email{michele.alde@asc.tuwien.ac.at}
\author{Hywel~Normington}
\address{Department of Mathematics,
University of Manchester,
Alan Turing Building, Oxford Road, Manchester M13~9PL, United Kingdom}
\email{hywel.normington@manchester.ac.uk}
\author{Dirk~Praetorius}
\address{TU Wien, Institute of Analysis and Scientific Computing
(E 101), Wiedner Hauptstraße 8--10,
1040 Vienna, Austria}
\email{dirk.praetorius@asc.tuwien.ac.at}
\author{Michele~Ruggeri}
\address{Department of Mathematics,
University of Bologna,
Piazza di Porta San Donato 5, 40126 Bologna, Italy}
\email{m.ruggeri@unibo.it}



\begin{abstract}
We propose and analyze a fully discrete numerical scheme to approximate solutions
to the inertial Landau--Lifshitz--Gilbert (iLLG) equation,
a state-of-the-art model to describe the dynamics of the magnetization
at subpicosecond time scales.
The method employs first-order finite elements for the spatial discretization
and the two-step backward differentiation formula (BDF2) for the time discretization, where the inertial term is treated by a convex combination of BDF2 and the backward Euler method.
We prove well-posedness and unconditional stability of the scheme,
and obtain explicit bounds for the error in the realization of the iLLG-inherent unit length constraint
satisfied by the magnetization.
Moreover,
we derive the discrete energy law satisfied by the approximations,
which resembles the inherent dissipative energy law of the model.
Finally, with a constructive approach, we prove unconditional convergence
of the approximations toward a weak solution of the problem.
Numerical experiments assess the performance of the scheme
and validate the theoretical findings.
\end{abstract}
\maketitle


\section{Introduction}
In the past few years, the study of ultrafast magnetization dynamics
has attracted a lot of attention, primarily due to its
potential use in future spintronic devices~\cite{dieny2020opportunities,walowski2016perspective},
along with recent confirmation of
nutational spin dynamics in
ferromagnetic materials~\cite{inertia2020,unikandanunni2022inertial,de2025magnetic},
theoretically predicted
some years ago~\cite{ciornei2011magnetization,wegrowe2012magnetization}.
The field of ultrafast magnetization dynamics was motivated by
a discovery in ferromagnetic nickel~\cite{beaurepaire1996ultrafast}, and 
the body of research has grown substantially
over the last two decades, e.g.,~\cite{koopmans2000ultrafast,
stanciu2007all,
lambert2014all,
mondal2021nutation,
dhali2024theory}.

The model commonly used to describe these ultrafast dynamics is a modification of the classical
Landau--Lifshitz--Gilbert (LLG) equation~\cite{landau1935lifshitz,Gilbert2004}: the so-called inertial Landau--Lifshitz--Gilbert (iLLG)
equation~\cite{ciornei2011magnetization,wegrowe2012magnetization}.
The order parameter is the magnetization, which is subject to a unit length constraint,
making the problem nonconvex and highly nonlinear.
The iLLG equation includes an additional term related
to the second time derivative of the magnetization,
and introduces nutation dynamics, i.e., rapid oscillations of the magnetization superimposed
on the slower damped precession~\cite{ciornei2011magnetization,mondal2023inertial,de2025magnetic}.

\subsection{Literature review}
From the analysis point of view, existence
of weak solutions has been proven in~\cite{hadda2014magnetization},
and more recently constructively in~\cite{ruggeri2022numerical} using two finite element schemes.
In terms of numerical methods, the number of reliable schemes is still limited in the literature.
We mention the finite difference schemes of~\cite{moumni2016finite,li2021second},
but of these, only~\cite{moumni2016finite} provides some numerical analysis in terms of
the conservation properties of the method. More recently,~\cite{chen2024convergence} proposes an implicit
finite difference scheme for the iLLG equation and establishes
second-order error estimates in both space and time, assuming a sufficiently regular exact solution
and zero initial velocity.
The tangent plane scheme (TPS) and angular momentum method (AMM) of~\cite{ruggeri2022numerical}
are two fully discrete finite element methods for the iLLG equation, and the author proves
the convergence to a weak solution.
The TPS is based upon
the classical algorithm from~\cite{alouges2008a}
and is linear,
but suffers from significant damping and, due to the application of a
nodal projection step, requires weakly acute meshes,
whereas the AMM is a nonlinear approach based upon a method for wave maps~\cite{karper2014new,bartels2015fast},
which is excellent for conservation properties,
but requires the solution of a nonlinear system at each time step.
An alternative nonlinear method has been proposed in~\cite{d2024midpoint},
based upon the midpoint scheme for the LLG equation~\cite{d2005geometrical},
which again has excellent conservation properties,
even in the conservative case, but requires the
solution of a nonlinear system at each time step. The midpoint scheme is not analyzed
in~\cite{d2024midpoint}.
We further mention that
the midpoint schemes of~\cite{d2024midpoint,chen2024convergence}
and the AMM of~\cite{ruggeri2022numerical} are formally of second order in time,
the method of~\cite{li2021second} is formally of second order in time and space,
while the TPS of~\cite{ruggeri2022numerical} is only of first order in time.
Recently, a numerical scheme for the iLLG equation including Maxwell's equations has been
proposed and analyzed~\cite{kitengeso2025convergent}.

The solvability of the nonlinear systems arising in the methods from~\cite{d2024midpoint,ruggeri2022numerical,chen2024convergence}
is typically ensured via a fixed-point iteration; however, a Courant--Friedrichs--Lewy (CFL)-type condition is required to guarantee unique solvability.
For the AMM, a weak CFL condition of the form $k = O(h)$ is required,
but for the midpoint scheme of~\cite{chen2024convergence},
a stronger CFL condition of the form $k = O(h^2)$ is needed.
Here, \(h\) and \(k\) denote the mesh size and the time-step size, respectively.

For the classical LLG equation, i.e., without inertial effects, the numerical analysis is by now well developed. Existing works can roughly be divided into schemes which converge with rates towards sufficiently smooth strong solutions~\cite{cimrak2005,bs2006,gao2014,an2016,afkl2021,als2025,afp2026_partII} (see, e.g.,~\cite{cf2001,feischl2017regular} for the existence of such regular solutions) and schemes which guarantee convergence, along subsequences, towards weak solutions~\cite{aj2006,bp2006,bkp2008,alouges2008a,kw2018,afp2026_partI}. Among the latter, the TPS of~\cite{aj2006,alouges2008a} requires only one linear system per time step, as opposed to, e.g., the nonlinear midpoint scheme of~\cite{bp2006}, but is only of first order in time.
Previous TPSs have been proposed for the LLG equation
that remove the nodal projection step,
thereby removing the need for weakly acute meshes~\cite{abert_spin-polarized_2014,bartels2016}.
This, however, comes at the cost of losing exact preservation of the unit length constraint,
and instead allows for a unit length constraint
violation that is proportional to the time-step size.
For such projection-free TPSs, convergence towards weak solutions is shown in~\cite{abert_spin-polarized_2014,hpprss2019}, while a priori error estimates for sufficiently smooth solutions of the LLG equation coupled with the eddy current equations are derived in~\cite{feischl2017eddy}. Convergent finite element schemes for the LLG equation coupled with the full or quasi-static Maxwell equations, as well as with the eddy current equations, are analyzed, e.g., in~\cite{bbp2008,lt2013,bpp2015,lppt2015}. Higher-order linear schemes include the almost second-order TPSs of~\cite{akst2014,dpprs2019}, the linearly implicit BDF methods of order up to five of~\cite{afkl2021}, and second-order BDF-type schemes, both for finite elements~\cite{yj2021,afp2026_partI,afp2026_partII} and finite differences~\cite{xie2020second}.
Recent gradient flow-based iterative methods for harmonic maps have been proposed that
improve the unit length constraint violation
to being proportional to the square of the time-step size,
provided some discrete time
regularity condition is satisfied~\cite{akrivis2025quadratic,akrivis2025midpoint,afp2026_partI,dong2024bdf}.

\subsection{Contributions}
In this work,
following~\cite{akrivis2025quadratic,afp2026_partI},
we propose a fully discrete integrator for the iLLG equation
with a time discretization based on the two-step backward differentiation formula (BDF2).
More precisely,
we extend a projection-free version of the TPS of~\cite{ruggeri2022numerical},
with a BDF2 update
for the magnetization, and an extrapolation for the tangent space.
To treat the inertial
contribution to the iLLG equation,
we apply a \(\beta\)-weighted sum of BDF2 and backward Euler time-steppings
for a parameter \(\beta\in[0,1]\), where setting \(\beta = 0\) yields a full BDF2 method.
We prove well-posedness of the resulting integrator,
we show that the approximations satisfy a discrete energy law which mimics that of the continuous problem,
and we prove the unconditional stability of the method (in the sense that
there are no CFL conditions relating the mesh and time-step parameters).
The latter,
combined with the improved time regularity of the iLLG equation,
allows us to
eliminate the discrete regularity condition mentioned above,
yielding unconditional
second-order accuracy in the unit length constraint (if $\beta\in(0,1]$).
In line with the accelerated gradient descent methods of~\cite{dong2024bdf},
the unit length constraint
violation can even be shown to be of third order for all $\beta\in[0,1]$, subject to a discrete time regularity condition.
Furthermore,
we provide a proof of unconditional weak convergence of the approximations 
to a weak solution of the iLLG equation,
in line with~\cite{ruggeri2022numerical}.
Moreover, as an important contribution, we extend to the iLLG equation
the argument introduced for the LLG equation in~\cite{afp2026_PartIII}, and pass to the limit in the discrete energy inequality without imposing any CFL-type restriction. 
This improves \cite{afp2026_partI}, which required a CFL condition in the first and last time steps to prove the energy inequality. As a consequence, our integrator is fully unconditional and does not require any coupling between the meshsize $h$ and the time-step size $k$.
Numerical experiments confirm that the method yields unconditional second-order accuracy
in the unit length constraint, which improves to
third-order under a discrete time regularity condition.
We also provide numerical evidence that setting \(\beta = 0\)
yields second-order accuracy in time in the $\bH^1$-norm.

\subsection{Outline}
In Section~\ref{sec:model}, we state the initial boundary value problem for the iLLG equation
and the definition of its weak solution. In Section~\ref{sec:ingredients},
we give the necessary notation and ingredients for introducing our method.
In Section~\ref{sec:main}, we state our fully discrete algorithm, and remark on the
weighting parameter \(\beta\).
In Section~\ref{sec:numerics}, we assess the performance of the scheme
and validate our theoretical results numerically via
a series of experiments. In particular, we demonstrate the
order of the unit length constraint, and the order of accuracy in time, along with the effects
of the \(\beta\) parameter.
Lastly, in Section~\ref{sec:proofs}, we collect the proofs of all of our results.

\subsection{General notation.}
Without any ambiguity, we denote by $|\cdot|$ the measure of a measurable set in $\R^d$, the absolute value of a scalar in $\R$, 
and the Euclidean norm of a vector in $\R^d$.
We employ standard conventions for Lebesgue, Sobolev, and Bochner spaces throughout this work,
adopting boldface notation for vector-valued and matrix-valued function spaces,
such as $\bL^2(\Omega) \coloneqq L^2(\Omega;\R^3)$ and $\bH^1(\Omega)\coloneqq H^1(\Omega;\R^3)$.
For $\uu, \vv \in \bL^2(\Omega)$, the $L^2$-inner product over $\Omega\subset\R^d$, $d=2,3$, is defined as
\begin{align*}
 \langle\uu,\vv\rangle 
 \coloneqq 
 \langle\uu,\vv\rangle_{\bL^2(\Omega)}= \int_\Omega\uu(\xx)\cdot\vv(\xx)\,\mathrm{d}\xx,
\end{align*}
with the corresponding norm denoted by $\norm{\cdot}$, i.e., $\norm{\uu} \coloneqq \langle\uu,\uu\rangle^{1/2}$.
A different convention applies to the space-time domain $\OT\coloneqq \Omega \times [0,T]$,
where we denote the scalar product and norm by $\langle\cdot,\cdot\rangle_{\OT}$ and $\norm[\OT]{\cdot}$, respectively.
All remaining norms are explicitly specified when introduced.

By $C > 0$ we denote a generic constant, which is always independent of the discretization parameters,
and may differ at each use. We adopt the shorthand $A\lesssim B$ to mean $A\leq C \, B$
for some constant $C > 0$ whose interpretation is clear from context, and always independent of the discretization parameters \(h\) and \(k\).
Additionally, $A \simeq B$ indicates the two-way inequality $A \lesssim B \lesssim A$.


\section{Model problem}\label{sec:model}
Consider a bounded Lipschitz domain $\Omega\subset \R^d$, with $d\in\{2,3\}$, 
describing a ferromagnetic body.
At subpicosecond time scales,
the ultrafast dynamics of the magnetization
$\mm:\Omega\times(0,\infty)\to \mathbb{S}^2\coloneq\{\boldsymbol{x}\in \R^3 \colon |\boldsymbol{x}|=1\}$
is governed by the following initial boundary value problem for the iLLG equation:
\begin{subequations}\label{eq: iLLG system}
\begin{empheq}[left=\empheqlbrace]{align}
    &\mmt = -\mm \times \heff[\mm]
    + \alpha \mm \times \mmt
    + \tau \mm \times \mmtt && \text{in} \quad \Omega \times (0,\infty),\label{eq:iLLG}\\
    &\partial_{\nn}\mm = \0 && \text{on} \quad \partial\Omega \times (0,\infty),\label{eq:neumann_boundary} \\
    &\mm(0) = \mm^0 && \text{in} \quad \Omega,\label{eq:initial_condition_m}\\
    &\mmt(0) = \vv^0 && \text{in} \quad \Omega.\label{eq:initial_condition_v}
\end{empheq}
\end{subequations}
In~\eqref{eq: iLLG system} (stated in dimensionless form;
see, e.g.,~\cite[Remark~2.2]{ruggeri2022numerical} for details on the rescaling),
\(\alpha>0\) denotes the Gilbert damping parameter
and $\tau > 0$ is the angular momentum relaxation time.
The right-hand side of~\eqref{eq:iLLG} includes,
beyond the classical precession and damping terms,
an additional term with the second time derivative of the magnetization.
This extra term is responsible for nutation dynamics:
rapid oscillations of the magnetization superimposed on the slower damped precession,
occurring at much higher frequencies. Note that setting formally 
\(\tau=0\) in~\eqref{eq:iLLG},
one recovers the classical LLG equation.

The initial conditions satisfy the unit length constraint \(|\mm^0|=1\)
and the compatibility condition \(\mm^0\cdot\vv^0 = 0\).
Note that for \(t>0\) it holds that
\[
    \frac{\mathrm{d}}{\mathrm{d} t}|\mm(t)|^2 = 2\mm(t)\cdot\mmt(t) \stackrel{\eqref{eq:iLLG}}{=} 0.
\]
Therefore, we see that the magnitude of \(\mm\) is conserved in time,
as from \(|\mm^0| = 1\) it follows that \(|\mm(t)| = 1\) for all \(t>0\).
In this work, we consider only the highest order contribution to the energy, namely the exchange energy, given by
\begin{equation}\label{eq: exchange energy}
    \E[\mm] = \frac{1}{2}\int_{\Omega} |\Grad\mm|^2\,\mathrm{d}\xx.
\end{equation}
The negative Gâteaux derivative of the energy functional~\eqref{eq: exchange energy} with respect to the magnetization 
is called the \emph{effective field} \(\heff\) and is the quantity that drives the magnetization dynamics of~\eqref{eq:iLLG}, responsible for the
precession term. In strong form, it is given by
\begin{equation*}\label{eq:effective_field}
    \heff[\mm] = -\frac{\delta\E[\mm]}{\delta\mm} = \Lapl \mm.
\end{equation*}
As opposed to the LLG equation, the iLLG equation is also
associated with a \emph{kinetic energy}, given by $\tau\norm{\mmt}^2/2$,
which motivates the definition of the extended energy functional
\begin{equation}\label{eq: Energy functional}
    \mathcal{J}[\mm,\uu] = \E[\mm] + \frac{\tau}{2}\norm{\uu}^2.
\end{equation}
Within the kinematic framework, the quantity $\mathcal{J}[\mm,\partial_t\mm]$ can be interpreted as
the total energy of the magnetization, consisting of the potential energy $\E[\mm]$ and the kinetic energy $\tau\norm{\partial_t\mm}^2/2$.
Sufficiently smooth solutions to the iLLG equation satisfy a dissipative energy law
\begin{equation}\label{eq:strong_energy_law}
    \frac{\mathrm{d}}{\mathrm{d} t}\mathcal{J}[\mm,\mmt] = -\alpha\norm{\mmt}^2 \leq 0,
\end{equation}
which is interpreted as the total energy of the system being non-increasing in time,
where the dissipation is modulated by the Gilbert damping parameter \(\alpha\).
As in~\cite{ruggeri2022numerical}, we consider weak solutions 
of~\eqref{eq: iLLG system} in the following sense:
\begin{definition}[Global weak solution]\label{def: weak solution}
    Let $\mm^0\in\bH^1(\Omega;\mathbb{S}^2)$
    and \(\vv^0\in\bL^2(\Omega)\)
    such that \(\mm^0\cdot\vv^0 = 0\) a.e.\ in \(\Omega\).
    A vector field \(\mm:\Omega\times(0,\infty)\to\R^3\) is called
    a \emph{global weak solution} of~\eqref{eq: iLLG system}
    if \(\mm\in L^\infty(0,\infty;\bH^1(\Omega))\cap W^{1,\infty}(0,\infty;\bL^2(\Omega))\)
    and, for all \(T>0\), the following properties are satisfied:
    \begin{enumerate}[label={\rm(\roman*)}, ref={\rm\roman*}]
        \item\label{def: weak solution i} $\mm\in\bH^1(\Omega_T)$ satisfies $|\mm|=1$ a.e.\ in $\OT$;
        \item\label{def: weak solution ii} \(\mm(t) \to \mm^0\) in $\bH^1(\Omega)$ and \(\partial_t\mm(t) \to \vv^0\) in $\bL^2(\Omega)$ as $t \to 0$;
        \item\label{def: weak solution iii}
        For all \(\vvarphi\in C_{c}^{\infty}([0,T);\bH^1(\Omega))\), it holds that 
        \begin{equation}\label{eq:weak_solution var_formulation}
            \begin{split}
            \int_{0}^{T}\inner{\mmt(t)}{\vvarphi(t)}\dtt &=
            - \int_{0}^{T}\inner{\heff[\mm(t)]}{\vvarphi(t)\times\mm(t)}\dtt
            + \alpha\int_{0}^{T}\inner{\mm(t)\times\mmt(t)}{\vvarphi(t)}\dtt\\
            &- \tau\int_{0}^{T}\inner{\mm(t)\times\mmt(t)}{\partial_t\vvarphi(t)}\dtt
            - \tau\inner{\mm^0\times\vv^0}{\vvarphi(0)};
            \end{split}
        \end{equation}
        \item\label{def: weak solution iv} It holds that
        \begin{equation}\label{eq:weak_energy_law}
            \mathcal{J}[\mm(T),\mmt(T)]
            + \alpha \int_0^T\norm{\mmt(t)}^2 \dtt
            \leq \mathcal{J}[\mm^0, \vv^0].
        \end{equation}
    \end{enumerate}
\end{definition}
Definition~\ref{def: weak solution}\eqref{def: weak solution iii} is based on a
variational formulation of~\eqref{eq: iLLG system} over the space-time domain \(\OT = \Omega \times [0,T]\).
Equation~\eqref{eq:weak_solution var_formulation} is derived by
integrating by parts in time, which reduces the regularity requirements on \(\mm\).
In particular, we use the elementary identity
\[
    \partial_t(\mm\times\partial_t\mm) = \partial_t\mm\times\partial_t\mm 
                                        + \mm\times\partial_{tt}\mm
                                        = \mm\times\partial_{tt}\mm.
\]
The term involving the effective field $\heff[\mm]=\Lapl \mm$ is understood as
\[
    \inner{\heff[\mm(t)]}{\vvarphi(t)\times\mm(t)} = \sum_{j=1}^d\inner{\partial_{j}\mm(t)}{\mm(t)\times\partial_j\vvarphi(t)},
\]
where, for sufficiently smooth functions, the identity follows from integration by parts in space and the homogeneous Neumann boundary condition.
Definition~\ref{def: weak solution}\eqref{def: weak solution iv} is a
weak energy law,
being the weak counterpart of the energy law~\eqref{eq:strong_energy_law}.
In Definition~\ref{def: weak solution}\eqref{def: weak solution ii} and~\eqref{def: weak solution iv}, $\mmt$ is evaluated pointwise in time. This is understood for the representative of $\mmt$ that is weakly continuous in time with values in $\bL^2(\Omega)$; such a representative exists for every $\mm\in L^\infty(0,\infty;\bH^1(\Omega))\cap W^{1,\infty}(0,\infty;\bL^2(\Omega))$ satisfying Definition~\ref{def: weak solution}\eqref{def: weak solution i} and~\eqref{def: weak solution iii}; see Substep~3d of the proof of Theorem~\ref{thm: main result}.

\section{Preliminaries}\label{sec:ingredients}

In this section, we collect some notation and preliminary results that will be necessary
to introduce and analyze the fully discrete algorithm we propose to approximate solutions
to the initial boundary value problem~\eqref{eq: iLLG system}.

\subsection{Time discretization}

Given the final time $T$, we consider time-step sizes $k>0$ such that $T=Nk$ for some $N\in \N$. We consider the uniform `partition' of the time interval $[0,T]$
given by $t_j \coloneqq jk$ for all $j \ge 0$.
Given a sequence $\{\uu^j\}_{j \ge 0}$ such that $\uu^j \approx \uu(t_j)$ and $\vv^0 \approx \partial_t \uu(t_0) = \partial_t\uu(0)$
for some vector-valued function $\uu$ defined on $[0,\infty)$,
we define the difference quotients
\begin{align*}
\dt \uu^{j+1} &\coloneqq
\frac{\uu^{j+1} - \uu^{j}}{k}, \quad \text{if } j \ge 0,\\
\dt^2 \uu^{j+1}
&\coloneqq \frac{\dt \uu^{j+1} - \dt{\uu}^j}{k}
=
\frac{\uu^{j+1} - 2\uu^j+\uu^{j-1}}{k^2}, \quad \text{if } j \ge 1.
\end{align*}
Furthermore, we consider the approximations $\vv^j \approx \partial_t \uu(t_j)$ given by
\begin{equation*}
\vv^{1} \coloneqq \dt \uu^1
\quad
\text{and}
\quad
\vv^{j+1} \coloneqq \frac{1}{2k}(3\uu^{j+1} - 4 \uu^{j} + \uu^{j-1})
\quad \text{for all } j \ge 1,
\end{equation*}
based on the backward Euler method
and the BDF2 method, respectively.
Note that
\begin{equation*}
\uu^1 = \uu^0 + k \vv^1,
\quad
\text{and}
\quad
\uu^{j+1} = \frac{1}{3}(4\uu^{j} - \uu^{j-1} + 2k\vv^{j+1})
\quad \text{for all } j \ge 1.
\end{equation*}
We also define the extrapolations
\begin{equation*}
\widehat\uu^1 = \uu^0,
\quad
\text{and}
\quad
\widehat\uu^{j+1} = 2\uu^{j} - \uu^{j-1}
\quad \text{for all } j \ge 1.
\end{equation*}
Consider the pairs given by $\mathcal{U}^{j+1}\coloneq(\uu^{j+1},\uu^{j})$ for all $j \ge 0$.
Define the symmetric matrix $G=(g_{ij})_{i,j=1,2}$
with entries $g_{11}=5/4,\,g_{12}=-1/2,\,g_{22}=1/4$.
For all pairs, we consider the BDF2-adapted $G$-norm (see, e.g.,~\cite{Dahlquist} or~\cite[Section~V.6]{hairer})
\begin{equation}\label{eq:G-norm}
    \begin{split}
    |\mathcal{U}^{j+1}|^2_G = (G\mathcal{U}^{j+1})\cdot\mathcal{U}^{j+1}
                            &= g_{11}|\uu^{j+1}|^2 + 2g_{12}\uu^j\cdot\uu^{j+1}
                            + g_{22}|\uu^{j}|^2\\
        &=\frac14(|\uu^{j+1}|^2 + |2\uu^{j+1}-\uu^{j}|^2).
    \end{split}
\end{equation}
Since $G$ is a positive-definite matrix with eigenvalues
$\mu_{\pm} = (3\pm 2\sqrt{2})/4$, there holds the norm equivalence
\begin{equation}\label{eq: G-norm-equivalence}
    \mu_{-}(|\uu^{j+1}|^2+|\uu^{j}|^2) \le |\mathcal{U}^{j+1}|^2_G
                             \le \mu_{+}(|\uu^{j+1}|^2+|\uu^{j}|^2).
\end{equation}
Moreover,
the $G$-norm satisfies the identities (cf.~\cite{akrivis2025quadratic,afp2026_partI})
\begin{align}\label{eq:G_inner_product}
    \vv^{j+1} \cdot \uu^{j+1} &= \dt|\mathcal{U}^{j+1}|_G^2
        + \frac{k^3}{4}|\dt^2\uu^{j+1}|^2\quad \text{for all } j \ge 1,\\\label{eq:G_norm_difference}
    |\mathcal{U}^{j+1}|^2_G  &= \frac{3}{4}|\uu^{j+1}|^2
                        - \frac{1}{4}|\uu^{j}|^2
                        + \frac{1}{2}|\uu^{j+1}-\uu^{j}|^2\quad \text{for all } j \ge 0.
\end{align}

For \(\{\phi^j\}_{j\ge0}\),
we define the time reconstructions
$\phi_{k}$, $\phi^-_{k}$, $\phi^+_{k}$,
defined, for all $j\ge0$ and $t \in [t_j,t_{j+1})$, by
\begin{equation}\label{eq:reconstructions}
\phi_{k}(t) \coloneqq \frac{t-t_j}{k}\phi^{j+1} + \frac{t_{j+1} - t}{k}\phi^j,
\quad
\phi_{k}^-(t) \coloneqq \phi^j,
\quad
\phi_{k}^+(t) \coloneqq \phi^{j+1}.
\end{equation}
Note that
$\partial_t \phi_k (t) =  \dt \phi^{j+1}$
for all $t \in (t_j,t_{j+1})$.

\subsection{Space discretization}
For the space discretization, let $\{\T_h\}_{h>0}$ be a $\gamma$-quasi-uniform family of meshes
of $\Omega$ into triangles (if $d=2$) or tetrahedra (if $d=3$), i.e.,
\begin{align*}
 \gamma^{-1} h \le |K|^{1/d} \le \diam(K) \le h
 \quad \text{for all } K \in \T_h \text{ and all } h > 0,
\end{align*}
where $h_K = \diam(K)$ for all $K \in \T_h$, and
$h=\max_{K\in\T_{h}}h_K$ is the mesh size of $\T_{h}$.
We denote by $\NN_{h}$ the set of nodes in the mesh $\T_{h}$.
For all $K \in \T_h$,
we denote by $\mathcal{P}_{1}(K)$ the space of
polynomials of degree at most $1$ on $K$.
We denote by $\S^1(\T_{h})$ the space of piecewise
affine and globally continuous functions from $\Omega$ to $\R$,
i.e.,
\begin{equation*}
\S^1(\T_{h}) = \{\phi_{h}\in C(\overline{\Omega})\colon\phi_{h}\vert_{K}\in \mathcal{P}_{1}(K)
\text{ for all } K\in\T_{h}\} \subset H^1(\Omega).
\end{equation*}
We denote by $\II_{h}\colon C(\overline{\Omega})\to \S^1(\T_{h})$ the nodal interpolant
satisfying $\II_{h}[\phi](\zz) = \phi(\zz)$ for each $\zz\in\NN_{h}$,
where $\phi$ is a continuous function.

Since the unknowns of the problem in which we are interested are vector fields,
we consider the vector-valued finite element
space $\boldsymbol{\S}^1(\T_{h}) = \S^1(\T_h)^3$
and denote the vector-valued nodal interpolant as
$\boldsymbol{\II}_{h}\colon\CC(\overline{\Omega})\to \boldsymbol{\S}^1(\T_{h})$.
For all $j \ge 0$ the approximate magnetization
$\mm_h^j \approx \mm(t_j)$ will be sought in the set
\begin{equation*}
\boldsymbol{\mathcal{M}}_{h,\delta}
= \big\{\pphi_{h}\in \boldsymbol{\S}^1(\T_{h})\colon
    \abs{\pphi_{h}(\zz)}\geq 1\text{ for all } \zz\in\NN_{h}\,
    \text{and}\,\norm[L^1(\Omega)]{\II_{h}\left[|\pphi_h|^2\right] - 1}
    \leq \delta \big\}
\end{equation*}
for some $\delta\mathrel{\ge}0$.
Note that discrete magnetizations in
$\boldsymbol{\mathcal{M}}_{h,\delta}$ generally do not satisfy the unit length constraint,
not even at the vertices of the mesh,
but the error is controlled in the $L^1$-sense by $\delta$.
The approximations constructed by algorithms in~\cite{abert_spin-polarized_2014,bartels2016, afp2026_partI, afp2026_partII} 
belong to this set with $\delta=\mathcal{O}(k)$.
For the case $\delta=0$, we obtain the set
\begin{equation*}
\boldsymbol{\mathcal{M}}_{h,0} = \{\pphi_{h}\in \boldsymbol{\S}^1(\T_{h})\colon\abs{\pphi_{h}(\zz)}=1\text{ for all } \zz\in\NN_{h}\},
\end{equation*}
in which the constraint holds at the vertices of the mesh.
Approximate magnetizations in this set are usually obtained by applying a nodal projection step
or designing nonlinear constraint-preserving time discretizations; see, e.g., \cite{alouges2008a,bkp2008,bp2006}.

Another important property of solutions to the iLLG equation is
the orthogonality $\mm\cdot\mmt = 0$.
The fundamental ingredient of the TPS to realise it at the discrete level,
given an approximation $\mm_h^j \approx \mm(t_j)$ in $\boldsymbol{\mathcal{M}}_{h,\delta}$,
is the discrete tangent space
\begin{equation*}
\bT_{h}(\mm_{h}^j)
    = \{\ppsi_{h}\in \boldsymbol{\S}^1(\T_{h})\colon \mm_{h}^j(\zz)\cdot\ppsi_{h}(\zz)
        = 0 \text{ for all } \zz\in\NN_{h}\},
\end{equation*}
where approximations $\vv_h^j \approx \mmt(t_j)$ will be sought.
Note that the desired orthogonality property is imposed only at the vertices of the mesh.

\section{Algorithm and main results}\label{sec:main}

We now present our fully discrete algorithm
to approximate solutions to the iLLG equation~\eqref{eq: iLLG system}.

\begin{algorithm}\label{algorithm}
\textbf{Input:}
Conforming mesh $\mathcal{T}_h$ of $\Omega$ with mesh size $h>0$,
time-step size $k>0$,
$\beta\in[0,1]$,
$\mm_h^0\in \boldsymbol{\mathcal{M}}_{h,0}$,
$\vv_{h}^{0} \in \bT_{h}(\mm_{h}^{0})$.
\begin{enumerate}[label=\rm (\roman*), ref=\rm \roman*]
\item\label{eq: first step full discr item} 
For $j=0$, repeat the following steps~\eqref{item: var form init}--\eqref{item: def m one}:
\begin{enumerate}[label=\rm(\alph*), ref=\rm \alph*]
    \item\label{item: var form init}
     Compute $\vv_h^1\in \bT_h(\mm_h^0)$ such that,
    for all $\pphi_h\in \bT_h(\mm_h^0)$,
    \begin{equation}\label{alg1:magnetization_update_first}
        \tau\inner{\dt\vv_h^1}{\pphi_h}
        +\alpha \inner{\vv_h^1}{\pphi_{h}}
        + \inner{\mm_{h}^{0}\times \vv_h^1}{\pphi_{h}}
        +  k\inner{\Grad\vv_h^1}{\Grad \pphi_{h}}
        = -\inner{\Grad\mm_h^0}{\Grad \pphi_{h}}
    \end{equation}
    \item\label{item: def m one}
     Set $\mm_{h}^{1} \coloneqq \mm_{h}^{0} + k \vv_h^1$.
\end{enumerate}
\item\label{item: next steps} 
For $j \ge 1$, repeat the following
steps~\eqref{item: def mhat}--\eqref{item: def mmhj}:
\begin{enumerate}[label=\rm(\alph*), ref=\rm \alph*]
\item\label{item: def mhat}
Set $\widehat{\mm}_h^{j+1}\coloneqq2\mm_h^{j}-\mm_h^{j-1}$.
\item\label{item: var form}
Compute $\vv_h^{j+1}\in \bT_h(\widehat{\mm}_h^{j+1})$ such that,
for all $\pphi_h\in \bT_h(\widehat{\mm}_h^{j+1})$,
\begin{equation}\label{alg1:magnetization_update}
    \begin{split}
    \beta\tau&\inner{\dt\vv_h^{j+1}}{\pphi_h}
    + (1-\beta)\frac{ \tau }{k}\inner{\frac{3}{2}\vv_h^{j+1} - 2\vv_h^{j} + \frac{1}{2}\vv_h^{j-1}}{\pphi_h}
    +\alpha \inner{\vv_{h}^{j+1}}{\pphi_{h}}\\
    &+ \inner{\mhat_h^{j+1}\times \vv_{h}^{j+1}}{\pphi_{h}}
    + \frac{2}{3}k\inner{\Grad\vv_h^{j+1}}{\Grad \pphi_{h}} 
    = -\frac{1}{3}\inner{\Grad(4\mm_h^j-\mm_h^{j-1})}{\Grad \pphi_{h}}.
    \end{split}
\end{equation}
\item\label{item: def mmhj}
Set $\mm_h^{j+1} \coloneqq \dfrac{4}{3}\mm_h^{j} - \dfrac{1}{3}\mm_h^{j-1} + \dfrac{2}{3}k \vv_h^{j+1}$.
\end{enumerate}
\end{enumerate}
\textbf{Output:} Sequences $\vv_h^j\approx\partial_t\mm(t_j)$
and $\mm_h^{j+1}\approx\mm(t_{j+1})$ for all $j \ge 0$.
\end{algorithm}

Note that in practical computations the number of iterations is not infinite,
but the algorithm stops when $jk \ge T$, where $T>0$ denotes some final time of interest.

Algorithm~\ref{algorithm} resembles the second-order BDF2
algorithm for the LLG equation given in~\cite{afp2026_partI},
but has an extra term due to the iLLG inertial term. 
To handle it,
a parameter \( \beta \in [0,1] \) is introduced
to split its discretization between two
components: a \(\beta\)-weighted first-order finite
difference derivative of \(\vv_h^{j+1}\),
and a \((1-\beta)\)-weighted second-order BDF2
finite difference derivative of \(\vv_h^{j+1}\).
This choice is motivated by complementary analytical
and numerical considerations:
as shown in Proposition~\ref{prop:stability} below,
the first-order discretization of the
inertial term ensures a bound on the quantity
\(
k^2\sum_{j=0}^{N-1}\norm{\dt \vv_h^{j+1}}^2,
\)
which is not guaranteed by the second-order BDF2 discretization alone.
Moreover, as shown later in~\eqref{eq: bound dt2 m j+1} in Lemma~\ref{prop:dtt_stability},
this bound is crucial in order to bound the quantity
\(
k^2\sum_{j=1}^{N-1}\norm{\dt^2\mm_h^{j+1}}^2,
\)
which is necessary to pass to the limit and
verify Definition~\ref{def: weak solution}(\ref{def: weak solution iii}),
as well as to guarantee unconditional second-order
accuracy in the unit length constraint error,
i.e., \eqref{eq:unit_length_constraint}.
Nevertheless, in the numerical experiments in Section~\ref{sec:numerics} we 
note that, in practice, using only the second-order BDF2
discretization (i.e., setting \(\beta = 0\)) still achieves
the second-order accuracy in the unit length constraint error. Also, 
incorporating the second-order BDF2 discretization yields, in our experiments, better
accuracy, especially regarding the energy decay and order of convergence.
The first-order discretization was used in~\cite{ruggeri2022numerical},
where it was noted that the tangent plane scheme suffered from
numerical dissipation when compared to a nonlinear scheme.
As the scheme in this work is formally of second order when \(\beta=0\),
this dissipation is less restrictive, but still present.
The use of the BDF2 discretization for the velocity further reduces this dissipation,
as will be shown in Section~\ref{sec:numerics}.
This comes at a technical cost, as the analysis
of convergence becomes much more involved.

In the following propositions,
we establish well-posedness,
energetic behavior,
stability,
and error bounds for the constraint violation.
The proofs are deferred to Section~\ref{sec:proofs}.

\begin{proposition}[well-posedness and discrete energy law]\label{prop: energy-estimate-1}
Each step of Algorithm~\ref{algorithm} is well posed. Moreover, the approximations generated by Algorithm~\ref{algorithm}
satisfy
\begin{equation} \label{eq:initial_energy_law}
    \mathcal{J}[\mm_h^1,\vv_h^1] + \alpha k \norm{\vv_h^1}^2 + \frac{\tau k^2}{2}\norm{\dt\vv_h^1}^2
     + \frac{k^2}{2}\norm{\Grad \vv_h^1}^2 = \mathcal{J}[\mm_h^0,\vv_h^0]
\end{equation}
and, for any $j\ge 1$,
\begin{align}
        &\mkern-15mu\mathcal{J}[\mm_h^{j+1},\vv_h^{j+1}] + \alpha k \norm{\vv_h^{j+1}}^2 + \frac{k^2}{2} \Big[\tau \norm{\dt \vv_h^{j+1}}^2
        +  \norm{\dt \Grad \mm_h^{j+1}}^2\Big] + \frac{k^4}{4}\norm{\dt^2\Grad\mm_h^{j+1}}^2\nonumber\\
        & \mkern+50mu+ \frac{k}{4} \dt\norm{\Grad\mm_h^{j+1}}^2
        + (1-\beta)\tau\Big[\frac{k^4}{4} \norm{\dt^2\vv_h^{j+1}}^2 + \frac{k}{4}\dt\norm{\vv_h^{j+1}}^2\Big] \label{eq:discrete_energy_equality}\\
        &\mkern-15mu= \mathcal{J}[\mm_h^j,\vv_h^j] + \frac{k}{4} \dt \norm{\Grad\mm_h^{j}}^2 
        + \frac{k^2}{2}\norm{\dt\Grad\mm_h^j}^2
        +(1-\beta) \tau \Big[\frac{k^2}{2} \norm{\dt\vv_h^{j}}^2 + \frac{k}{4}\dt\norm{\vv_h^{j}}^2\Big].\nonumber
\end{align}
Finally, for all $n \ge 2$, it holds that
\begin{equation}\label{eq:corollary_energy_law}
    \begin{split}
    &\mathcal{J}[\mm_h^n,\vv_h^n]
    + \alpha k\sum_{j=0}^{n-1}\norm{\vv_h^{j+1}}^2
    +\beta\frac{\tau k^2}{2}\sum_{j=0}^{n-1}\norm{\dt\vv_h^{j+1}}^2 
    +  \frac{k^2}{2}\norm{\dt\Grad\mm_h^n}^2\\
    & + (1-\beta)\tau \Big[\frac{k}{4}\dt\norm{\vv_h^n}^2
    +\frac{k^2}{2}\norm{\dt\vv_h^n}^2
    +\frac{k^4}{4}\sum_{j=1}^{n-1}\norm{\dt^2\vv_h^{j+1}}^2\Big]
    + \frac{k^4}{4}\sum_{j=1}^{n-1}\norm{\dt^2\Grad\mm_h^{j+1}}^2 \\
    &    + \frac{k}{4}\dt\norm{\Grad\mm_h^n}^2=
    \mathcal{J}[\mm_h^0,\vv_h^0]+ \frac{k}{4}\dt\norm{\Grad\mm_h^1}^2
    + (1-\beta)\frac{\tau k}{4}\dt\norm{\vv_h^1}^2.      
    \end{split}
\end{equation}
\end{proposition}

Looking at the discrete energy law~\eqref{eq:corollary_energy_law},
we see that, in addition to the iLLG-inherent $\alpha$-modulated energy dissipation of the total energy~\eqref{eq: Energy functional},
there appear positive terms on the left-hand side (yielding artificial dissipation) as well as unsigned perturbations
arising from the time discretization.

We make the following assumptions on the approximate initial data:
\begin{enumerate}[label=(A\arabic*), ref=A\arabic*]
\item\label{enum: A1} 
    $\mm_h^0 \in \boldsymbol{\mathcal{M}}_{h,0}$ satisfies
    $\mm_h^0 \to \mm^0$ in $\bH^1(\Omega)$ as $h\to 0$,
    which implies $\norm[\bH^1(\Omega)]{\mm_h^0}\lesssim 1$;
\item\label{enum: A2}
    $\vv_h^0 \in \bT_h(\mm_h^0)$ satisfies
    $\vv_h^0 \to \vv^0$ in $\bL^2(\Omega)$ as
    $h\to 0$, which implies $\norm{\vv_h^0}\lesssim 1$.
\end{enumerate}

\begin{proposition}[stability and constraint violation]\label{prop:stability}
Suppose that assumptions \eqref{enum: A1}--\eqref{enum: A2} are satisfied.
Let \(\beta \in[0,1]\).
For all $n \ge 2$,
the approximations generated by Algorithm~\ref{algorithm}
are uniformly bounded, in the sense that
\begin{equation}\label{eq:stability}
    \begin{split}
    \norm{\vv_h^{n}}^2
    + \norm[\bH^1(\Omega)]{\mm_h^{n}}^2
    &+ k\sum_{j=0}^{n-1}\norm{\vv_{h}^{j+1}}^2
    + \beta k^2\left(\sum_{j=0}^{n-1}\norm{\dt\vv_h^{j+1}}^2
    + \sum_{j=1}^{n-1}\norm{\dt^2\mm_h^{j+1}}^2\right)\\
    &+ (1-\beta)k^4\sum_{j=1}^{n-1}\norm{\dt^2\vv_h^{j+1}}^2
    + k^4\sum_{j=1}^{n-1}\norm{\dt^2\Grad\mm_h^{j+1}}^2
    \leq C_1.
    \end{split}
\end{equation}
It follows that
\begin{itemize}
\item
if \(\beta \in [0,1]\), then (unconditionally)
\begin{equation}\label{eq:unit_length_constraint_bad}
    \norm[L^1(\Omega)]{\II_{h}\big[|\mm_h^n|^2\big] - 1} \le C_2 k;
\end{equation}
\item
if \(\beta\in(0,1]\), then (unconditionally)
\begin{equation}\label{eq:unit_length_constraint}
    \norm[L^1(\Omega)]{\II_{h}\big[|\mm_h^n|^2\big] - 1} \le C_3 k^2.
\end{equation} 
\end{itemize}
Furthermore,
for any $\beta \in [0,1]$, if \(\vv_h^0 \equiv \0\) and the discrete regularity condition
\begin{equation}\label{eq: regularity condition}
k \norm{\dt\vv_h^1}^2 + k\sum_{j=1}^{n-1}\norm{\dt^2\mm_h^{j+1}}^2 \leq C_4
\end{equation}
holds for some \(C_4>0\) independent of $h$, $k$, and $n$,
then
\begin{equation}\label{eq:unit_length_constraint_third}
	\norm[L^1(\Omega)]{\II_{h}\big[|\mm_h^n|^2\big] - 1} \le C_5 k^3.
\end{equation}
The constants $C_1, C_2, C_3>0$ depend only on $\alpha$, $\tau$, the $\gamma$-quasi-uniformity of $\mathcal{T}_h$, and the constants on the initial data from \eqref{enum: A1}--\eqref{enum: A2}; in addition, $C_1$ depends on $|\Omega|$, and $C_3$ depends on $\beta$. The constant $C_5>0$ depends only on the $\gamma$-quasi-uniformity of $\mathcal{T}_h$ and on $C_4$.
\end{proposition}

\begin{remark}
Owing to stability~\eqref{eq:stability},
combined with the enhanced time regularity of the iLLG equation (compared to the one of the classical LLG equation),
we see that
the second-order unit length constraint violation is \emph{unconditional} if $\beta>0$.
    This contrasts with the methods proposed in~\cite{akrivis2025quadratic,afp2026_partI,akrivis2025midpoint}, where a discrete time regularity assumption 
    is already required to prove a second-order bound for the constraint violation. 
    In this respect, the present scheme is closer in spirit to the BDF2 method of~\cite{dong2024bdf},
    where a discrete regularity condition is only invoked to derive the improved third-order estimate.
    We emphasize, however, that also in our case the discrete regularity assumption~\eqref{eq: regularity condition}, 
    which is needed to obtain the third-order bound~\eqref{eq:unit_length_constraint_third}, 
    is not guaranteed by the scheme and has not been established analytically. 
    Nevertheless, since it depends solely on quantities computed during the iteration, 
    it can be evaluated and verified a posteriori numerically; see, e.g., the experiment of Section~\ref{sec:unitlengthconstraint}.
\end{remark}
\begin{remark}
    Proposition~\ref{prop:stability} shows that the scheme 
    is unconditionally stable for all $\beta \in [0,1]$, including the full BDF2 case $\beta = 0$. 
    In particular, $\beta\in(0,1]$ provides sufficient regularity and structure by~\eqref{eq:stability} 
    to carry out the full analysis, allowing one to establish not only stability but also
    convergence results and the unconditional second-order estimate~\eqref{eq:unit_length_constraint}, whereas for $\beta=0$ we can currently prove only the
    first-order estimate~\eqref{eq:unit_length_constraint_bad} unconditionally.
    The third-order estimate~\eqref{eq:unit_length_constraint_third} holds for all $\beta\in[0,1]$, but under the discrete regularity condition~\eqref{eq: regularity condition}.
    The numerical experiments show that the violation of the unit length
    constraint is of third order also for $\beta=0$.
    Moreover, for $\beta = 0$ we empirically observe second-order convergence in time of the error in the $\bH^1$-norm (see Section~\ref{sec:orderofconvergence}), although convergence for $\beta=0$ is not covered by our analysis.
\end{remark}

Using the approximations $\{\mm_h^j\}_{j \ge 0}$, $\{\widehat{\mm}_h^{j}\}_{j\ge 1},$ and $\{\vv_h^j\}_{j \ge 0}$ obtained with Algorithm~\ref{algorithm},
we define approximations $\mm_{hk}$, $\mm_{hk}^{\pm}$, $\widehat{\mm}_{hk}^+$, $\vv_{hk}^\pm$ of $\mm$ and $\vv=\partial_t\mm$ by \eqref{eq:reconstructions} and by
\begin{align}
\widehat{\mm}_{hk}^+(t) & \coloneqq
	\begin{cases}
		\mm_h^0 & \text{for } \, t\in[t_j,t_{j+1}) \quad\text{ with }\quad j=0,\\
		\widehat{\mm}_h^{j+1} & \text{for } \, t\in[t_j,t_{j+1}) \quad\text{ with }\quad j \ge 1.
	\end{cases} \label{eq:mhk-hat-plus}
\end{align}
We conclude the section by stating the following theorem,
which shows that the time reconstructions converge weakly toward a weak solution of~\eqref{eq: iLLG system} as $(h,k)\to(0,0)$.

\begin{theorem}[convergence]\label{thm: main result}
Let $\mathcal{T}_h$ be a family of $\gamma$-quasi-uniform meshes of $\Omega$.
Suppose that assumptions {\eqref{enum: A1}}--{\eqref{enum: A2}} are satisfied and suppose $\beta\in (0,1]$.
Then, the sequences of discrete functions $\mm_{hk}$, $\mm_{hk}^\pm$, and $\widehat{\mm}_{hk}^+$
admit (not relabeled) subsequences that unconditionally converge to a function
\(\mm\in L^\infty(0,\infty;\bH^1(\Omega))\cap W^{1,\infty}(0,\infty;\bL^2(\Omega))\),
which is a weak solution of iLLG~\eqref{eq: iLLG system} according to Definition~\ref{def: weak solution}\eqref{def: weak solution i}--\eqref{def: weak solution iv}.
More precisely, it holds that $\mm_{hk}\stackrel{*}{\rightharpoonup}\mm$ weakly-star in~$W^{1,\infty}(0,\infty;\bL^2(\Omega))$
and $\mm_{hk},\mm_{hk}^\pm,\widehat{\mm}_{hk}^+\stackrel{*}{\rightharpoonup}\mm$ in $L^\infty(0,\infty;\bH^1(\Omega))$ as $(h,k)\to(0,0)$,
where all convergences hold for the same subsequence.
\end{theorem}
Note that the existence of a weak-star limit satisfying the variational formulation of problem~\eqref{eq: iLLG system} and the energy inequality from Definition~\ref{def: weak solution}\eqref{def: weak solution iv} holds unconditionally, and, as in \cite{afp2026_PartIII}, no CFL-type condition is needed.
\section{Numerical experiments}\label{sec:numerics}

To demonstrate the improvements of Algorithm~\ref{algorithm}
over the first-order tangent plane
scheme introduced in~\cite{ruggeri2022numerical},
we present several numerical experiments.
For a fairer comparison,
we implement a modified version of the tangent
plane scheme of~\cite{ruggeri2022numerical},
where we remove the nodal projection step, as in the projection-free tangent 
plane scheme from~\cite{abert_spin-polarized_2014}.
This scheme practically coincides with repeating the step for $j=0$ for all time steps, and 
we refer to it as pfTPS1.
For comparison, we also plot the results obtained using the algorithm from~\cite{afp2026_partI} for the standard LLG equation. Note that this algorithm is naturally recovered from our inertial formulation (Algorithm~\ref{algorithm}) by setting $\tau=0$.
The 3D finite element scheme is implemented in NGSolve/Netgen~\cite{netgen},
and the orthogonality constraint of the discrete tangent space is implemented with
the nullspace method of~\cite{ramage2013preconditioned}.
To solve the linear system resulting from the finite element method
with the nullspace included, we apply the GMRES method from the SciPy
library, along with a sparse incomplete LU preconditioner, of drop tolerance \(10^{-4}\)
and fill factor of \(10\).
The tolerance for GMRES was set to \(10^{-15}\),
with convergence typically taking 4 iterations independently of
the experiment and \(\beta\).

The 2D experiment of Section~\ref{sec:blowup} is implemented in MATLAB.

\subsection{Unit length constraint}\label{sec:unitlengthconstraint}
As a first experiment, we investigate the order of convergence
in the unit length constraint for different values of \(\beta\in\{0,0.4,0.6,1\}\). 
Notice that \(\beta = 0\) is not covered by the convergence analysis nor by the second-order
unit length constraint~\eqref{eq:unit_length_constraint}, but only by the stability analysis, the first-order bound~\eqref{eq:unit_length_constraint_bad}, and, under the discrete regularity condition~\eqref{eq: regularity condition}, the third-order bound~\eqref{eq:unit_length_constraint_third}.
We choose \(\Omega = (0,1)^3\),
where the body is discretised with a mesh with
mesh size \(h_{\max} \approx 0.41\), consisting of 460 elements and 142 nodes.
We supplement the energy~\eqref{eq: exchange energy} with a Zeeman
external field \(\ff = (1,0,0)\), i.e., considering an extra term $-\int_\Omega \ff\cdot  \mm\,\mathrm{d}x$ in 
the energy definition.
We vary the time-step for \(n=0,\dots,10\) as \(k_n = 10^{-2}\cdot 2^{-n}\).
For the initial conditions, we choose $\vv_h^0 \equiv \0$ and
\begin{equation}\label{eq: high_energy_initial_condition}
    	\mm_h^0 \coloneq \boldsymbol{\II}_h[\mm^0] \, \text{ with } \, \mm^0(x,y,z) = \frac{2}{\sqrt{5}}\Big(\frac{1}{2}, \sin\big(4(x+y+z)\big), \cos\big(4(x+y+z)\big)\Big),
\end{equation}
yielding an initial total energy of \(\mathcal{J}[\mm_h^0,\vv_h^0]\approx 16\).
We fix \(\alpha=0.1\) and \(\tau = 0.1\).
The total time for this simulation is \(T = 1\).
To measure the unit length constraint error, we use
\begin{align}
	\text{err}_{L^1} &\coloneqq \norm[L^1(\Omega)]{\II_h[|\mm_h^N|^2] - 1},\label{eq: errL1}\\
	\text{err}_{\bL^\infty} &\coloneqq \norm[\bL^\infty(\Omega)]{\mm_h^N} - 1.\label{eq: errLinfty}
\end{align}
We also collect the auxiliary quantities related to the third-order
unit length constraint violation as in Proposition~\ref{prop:stability},
\begin{align}\label{eq:AB bounded quantities}
    A \coloneqq k \norm{\dt\vv_h^1}^2,\quad
    B \coloneqq k \sum_{i=1}^{N-1}\norm{\dt^2 \mm_h^{i+1}}^2.
\end{align}
We recall that, as \(\vv_h^0 \equiv \0\), if \(A\) and \(B\) are uniformly bounded,
then the unit length constraint violation is of third order in time,
while for pfTPS1, the unit length constraint
violation is always of first order in time.

\begin{figure}[ht]
\centering
\begin{subfigure}{0.49\textwidth}
    \includegraphics[width=\textwidth]{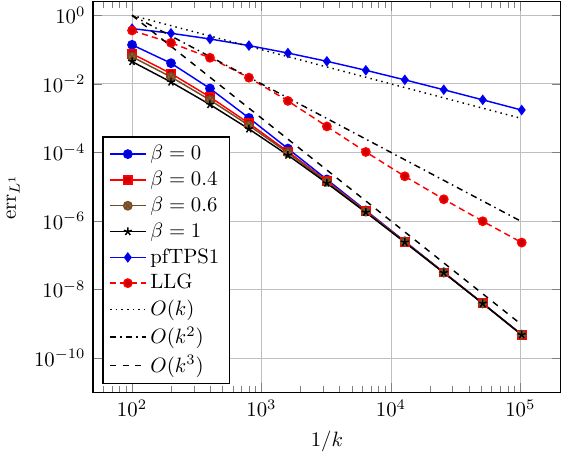}
    \caption{$\text{err}_{L^1}$ from \eqref{eq: errL1}.\label{subfig: unitlengthL1}}
\end{subfigure}
\hfill
\begin{subfigure}{0.49\textwidth}
    \includegraphics[width=\textwidth]{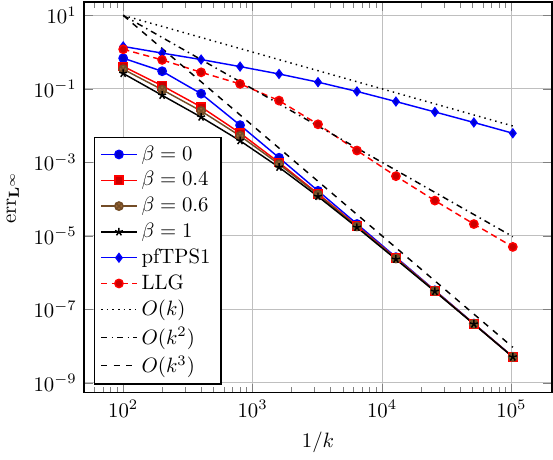}
    \caption{$\text{err}_{\bL^\infty}$ from \eqref{eq: errLinfty}.\label{subfig: unitlengthLinfty}}
\end{subfigure}
\caption{Experiment of Section~\ref{sec:unitlengthconstraint}.
Unit length constraint violation \(\text{err}_{L^1},\text{err}_{\bL^\infty}\)
vs \(1/k\), with varied \(\beta\).\label{fig:unitlengthconstraint}}
\end{figure}

\begin{figure}[ht]
\centering
\begin{subfigure}{0.49\textwidth}
    \includegraphics[width=\textwidth]{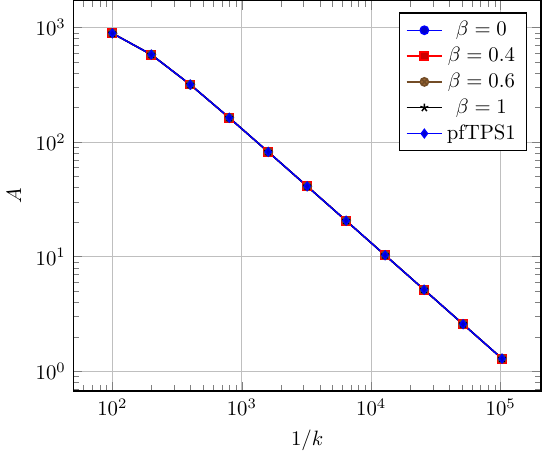}
    \caption{\(A\) from \eqref{eq:AB bounded quantities}.\label{subfig: unitlengthA}}
\end{subfigure}
\hfill
\begin{subfigure}{0.49\textwidth}
    \includegraphics[width=\textwidth]{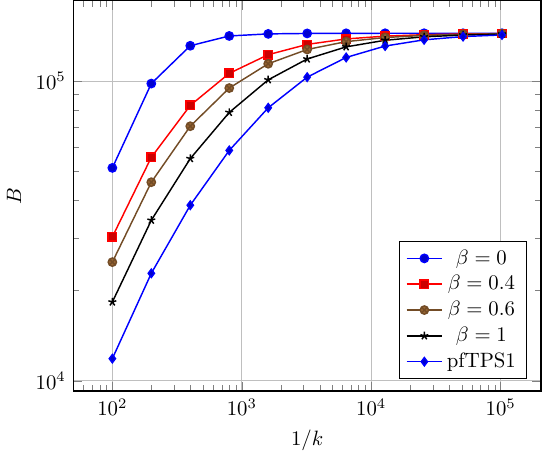}
    \caption{\(B\) from \eqref{eq:AB bounded quantities}.\label{subfig: unitlengthB}}
\end{subfigure}
\caption{Experiment of Section~\ref{sec:unitlengthconstraint}.
Quantities \(A, B\)
vs \(1/k\), with varied \(\beta\).\label{fig:unitlengthquantities}}
\end{figure}

The results are shown in Figure~\ref{fig:unitlengthconstraint}.
The second-order violation
in time as predicted by~\eqref{eq:unit_length_constraint}
is verified, even for \(\beta = 0\),
and we see that the third-order convergence
as predicted by~\eqref{eq:unit_length_constraint_third}
is attained for this example.
pfTPS1 performs significantly worse,
with only first-order constraint violation,
and would require a greatly reduced time-step size to perform similarly.
The LLG algorithm with \(\tau=0\) is of second order in time, but does not approach third-order
due to lacking the additional regularity of the iLLG equation.
If \(\beta\) is closer to \(1\), we see that the
unit length constraint is more precise,
which is because the second
derivative term 
\(k^2\sum_{j=1}^{N-1} \norm{\mathrm{d}_t^2 \mm_h^{j+1}}^2\)
responsible for the unconditional quadratic
consistency in~\eqref{eq:unit_length_constraint}
has a larger coefficient in the stability estimate~\eqref{eq:stability}.
In Figures~\ref{subfig: unitlengthA} and~\ref{subfig: unitlengthB}
we see that the quantities \(A\) and \(B\) are uniformly bounded.
Note that in Figure~\ref{subfig: unitlengthA}, all lines overlap exactly
because the initialization step is the same for all \(\beta\) and also for pfTPS1.
Figure~\ref{subfig: unitlengthB} shows convergence to an upper bound
of around \num{1.44E+05}, with \(\beta=0\) converging much faster than
for higher \(\beta\) values. We see that pfTPS1 also approaches the same limit,
but is much slower.

\subsection{Effects of the \texorpdfstring{\(\beta\)}{TEXT}-parameter}\label{sec:betaparameter}
We shall now demonstrate the effects of \(\beta\)
on the energy dissipation.
We apply the same mesh, and keep \(\tau=0.1\),
but consider the larger time-step \(k=10^{-2}\),
larger time interval \(T = 4\),
and a smaller Gilbert parameter \(\alpha = 10^{-2}\).
The constant initial conditions we consider are
\[
	\mm_h^0 \equiv \frac{10}{\sqrt{85}}(0.9,0.2,0)\quad \text{and} \quad \vv_h^0 \equiv \0.
\]
We supplement the energy with a Zeeman
external field \(\ff = (1,0,0)\).
We now plot the evolution in time of the
total energy \(\mathcal{J}[\mm_{hk}(t),\vv_{hk}^-(t)]\)
in Figure~\ref{fig:betatotalenergy},
and the average magnetization
in the \(x\)-direction, i.e.,
$\langle \mm_x(t) \rangle \coloneq 1/|\Omega|\int_\Omega\mm_{hk}(t)\cdot \ee_1\,\mathrm{d}\xx$
in Figure~\ref{fig:betaxavg}.
In Figure~\ref{fig:betatotalenergy}, the energy evolution is similar, but for larger \(\beta\)
we see increased dissipation. This is even more dramatic for pfTPS1,
with a different characteristic energy evolution,
being fully monotone decreasing. The energy evolution generated by Algorithm~\ref{algorithm}
is not generally monotone decreasing, especially early in the evolution for \(\beta=1\).
The effect of the additional dissipation on
the magnetization is clear in Figure~\ref{fig:betaxavg},
with pfTPS1 and \(\beta\) close to 1 yielding quickly decreasing amplitudes,
with pfTPS1 suffering the most.

\begin{figure}[ht]
	\centering
	\includegraphics[width=0.85\textwidth]{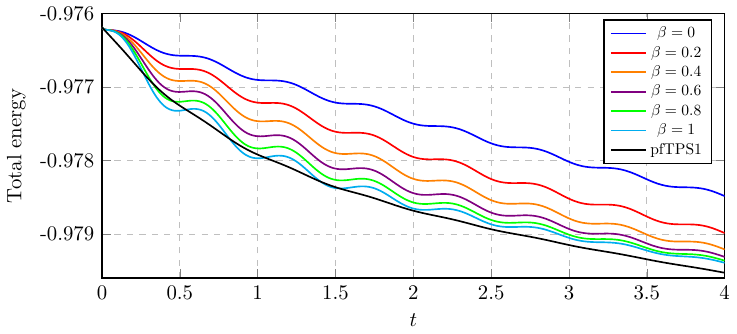} 
	\caption{Experiment of Section~\ref{sec:betaparameter}.
    Total energy \(\mathcal{J}[\mm_{hk}(t),\vv_{hk}^-(t)]\).\label{fig:betatotalenergy}}
\end{figure}

\begin{figure}[ht]
	\centering
	\includegraphics[width=0.85\textwidth]{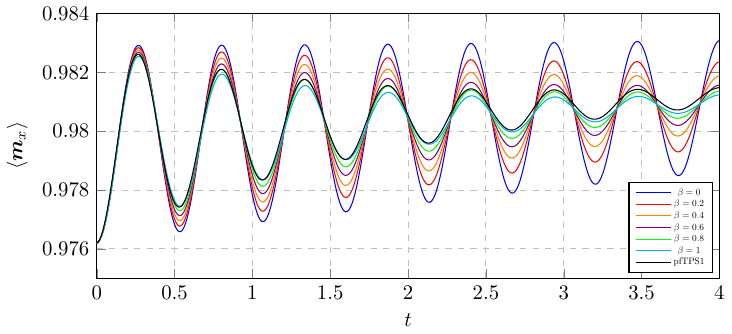} 
	\caption{Experiment of Section~\ref{sec:betaparameter}. Time-step of \(k = \num{1e-2}\).
	Average magnetization in \(x\)-direction, \(\langle \mm_x\rangle\).\label{fig:betaxavg}}
\end{figure}

\subsection{Convergence order in time}\label{sec:orderofconvergence}
To estimate the order of convergence, we
consider two experiments on a 3D triangulation of \(\Omega=(0,1)^3\). 
The mesh satisfies \(h_{\max}\approx 0.41\).
We set \(T=1/2\), \(\alpha=\tau=0.1\),
and consider the following two initial conditions.
We consider the constant initial condition
\[
	\mm_h^0 \equiv \frac{10}{\sqrt{85}}(0.9,0.2,0)\quad \text{and} \quad \vv_h^0 \equiv \0,
\] and the high energy initial condition~\eqref{eq: high_energy_initial_condition}.
We supplement the energy with a Zeeman external field \(\ff=(1,0,0)\) in both cases.
We consider time-step sizes \(k_n = \num{5e-3}\cdot 2^{-n}\) for \(n = 0,\dots,6\),
and compare this to a reference solution 
(computed for each value of \(\beta\))
corresponding to \(n=8\).
The error is computed with the \(\bH^1\)-norm
at the final time-step \(t=T\),
that is
\[
	\text{err}_{\bH^1}^n = \norm[\bH^1(\Omega)]{\mm_{h,k_n}(T) - \mm_{h,k_8}(T)}.
\]
In Figure~\ref{fig:orderofconvergence_3d}, we show the results with
both initial conditions and different values of \(\beta\).

We notice that $\beta = 0$ leads to second-order convergence, which is 
kept for very small values of $\beta$. As $\beta$ increases, the order 
of convergence decreases to first order. However, the estimated order
of convergence depends upon both the initial condition
and the value of \(\beta\). Indeed, notice that the constant 
initial condition leads to second-order convergence until $\beta = 10^{-3}$, and it starts to
drop to first order from $\beta = 5\times 10^{-3}$ (see Figure~\ref{fig:convergence_constant_3d}), while 
the situation is slightly better
for the high energy initial condition, where
we obtain second-order convergence until $\beta = 10^{-2}$, and it starts to 
drop to first order
from $\beta = 5\times 10^{-2}$ 
(see Figure~\ref{fig:convergence_high_energy_3d}), i.e., 
there is one order of magnitude of difference between the two cases.
Therefore, for $\beta=0$, the scheme is formally of second order as $k\to 0$,
while for any \(\beta>0\), the scheme is formally of first order.
The first-order approach, pfTPS1, is clearly seen to be first order in time for both experiments.

\begin{figure}[ht]
\centering
\begin{subfigure}{0.57\textwidth}
    \includegraphics[width=\textwidth]{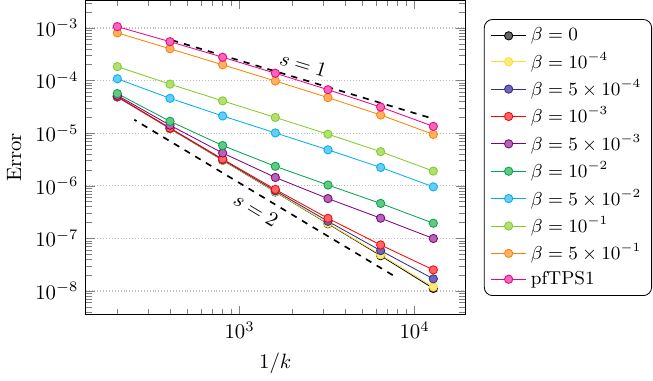}
    \caption{Constant initial condition}\label{fig:convergence_constant_3d}
\end{subfigure}
\begin{subfigure}{0.42\textwidth}
    \includegraphics[width=\textwidth]{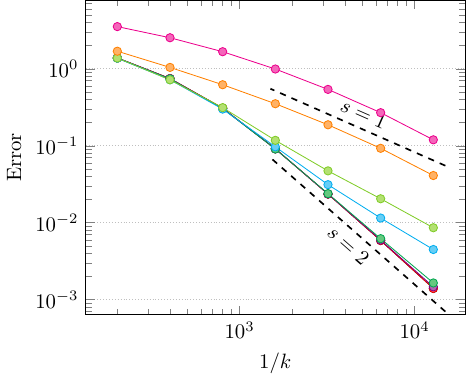}
    \caption{High energy initial condition}\label{fig:convergence_high_energy_3d}
\end{subfigure}
\caption{Experiment of Section~\ref{sec:orderofconvergence}.
\(\bH^1\)-norm of error \(\text{err}_{\bH^1}\)
vs. \(1/k\), with varied \(\beta\).
For small values of \(\beta\), we see second-order convergence,
but as \(\beta\) increases the order of convergence
drops to first order.}\label{fig:orderofconvergence_3d}
\end{figure}

To see the dependence of the order upon \(\beta\),
we instead plot the estimated order of convergence against \(\beta\),
using the two greatest values of \(1/k\) (i.e., $1/k=6400$ and $1/k=12800$) in Figure~\ref{fig:orderofconvergence_3d}.
The order of convergence is estimated using the formula
\[
    \text{Order of convergence} \approx -\log_2\left(\frac{\text{err}_{\bH^1}^6}{\text{err}_{\bH^1}^5}\right).
\]
The results are shown in Figure~\ref{fig:betaorderH1}, where we clearly have second-order convergence in time
for smaller values of \(\beta\).

\begin{figure}[ht]
    \centering
    \includegraphics{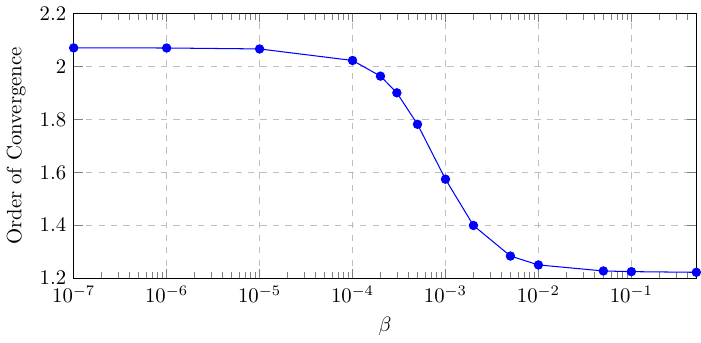}
    \caption{Experiment of Section~\ref{sec:orderofconvergence}.
    Estimated order of convergence against value of \(\beta\). We see that the estimated
    order of convergence approaches \(2\) as \(\beta\) becomes smaller.\label{fig:betaorderH1}}
\end{figure}

\subsection{Finite-time blow-up of weak solutions}\label{sec:blowup}
In this experiment we investigate the occurrence, for
a weak solution of iLLG, of the so-called finite-time
blow-up phenomenon, i.e., whether there exists a
time $T^*<\infty$ such that $\lim_{t\to T^*}\norm[\bL^\infty(\Omega)]{\Grad\mm(t)}=\infty$.
To this end, we consider the model problem
proposed in~\cite{ruggeri2022numerical},
which adapts to iLLG the model problem proposed in~\cite{bkp2008} for LLG.

We fix the 2D domain $\Omega = (-1/2,1/2)^2$, $\alpha = 1$, $\tau = 1$, and
final time $T=2$.
We consider homogeneous Neumann boundary conditions and,
given $\xx=(x_1,x_2) \in \Omega$, the initial data 
\[
    \mm^0(\xx)
    = \big(2 a(\xx) x_1, 2 a(\xx) x_2, a(\xx)^2 - |\xx|^2\big) / \big(a(\xx)^2 + |\xx|^2\big),
    \, \text{with}\, a(\xx) = \max\{0,(1-2|\xx|)\}^4,
\]
and $\vv^0 \equiv \0$.
We refer to~\cite{bkp2008} for the snapshots of
the evolution in time of this initial condition,
showing numerical approximations with large gradients.

Following~\cite{ruggeri2022numerical}, we fix $\ell\in\{4,5,6,7\}$ and
consider a uniform mesh $\mathcal{T}_{h_\ell}$
consisting of $2^{2 \ell+1}$
right-angled triangles, with
mesh size $h_{\ell} = \sqrt{2}\cdot 2^{-\ell}$.

\begin{figure}[ht!]
    \centering
\begin{subfigure}{0.49\textwidth}
    \includegraphics[width=\textwidth]{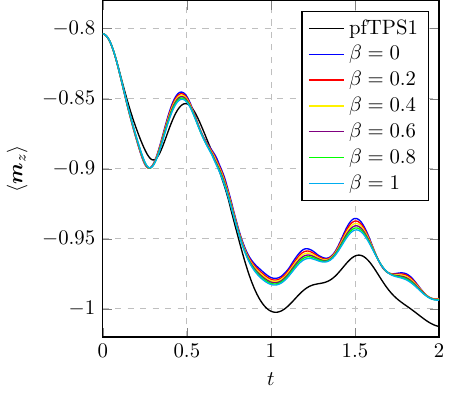}
    \caption{$k = h_5/10$}\label{subfig: ell5h10}
\end{subfigure}
\hfill
\begin{subfigure}{0.49\textwidth}
    \includegraphics[width=\textwidth]{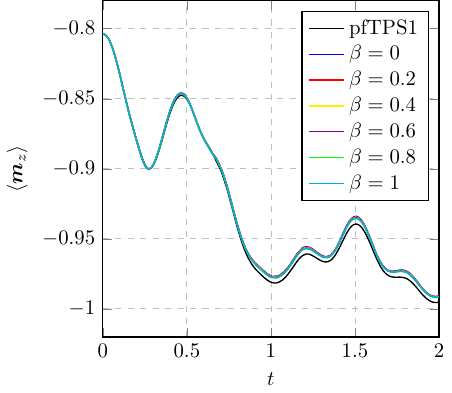}
    \caption{$k = h_5/100$}\label{subfig: ell5h100}
\end{subfigure}
\hfill
\vspace{0.5cm}
\begin{subfigure}{0.49\textwidth}
    \includegraphics[width=\textwidth]{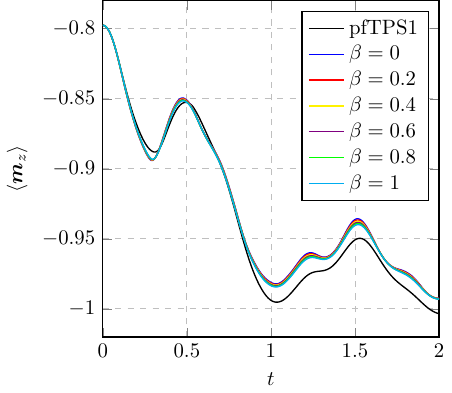}
    \caption{$k = h_6/10$}\label{subfig: ell6h10}
\end{subfigure}
\hfill
\begin{subfigure}{0.49\textwidth}
    \includegraphics[width=\textwidth]{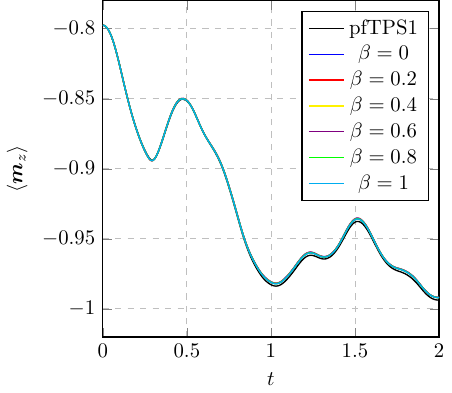}
    \caption{$k = h_6/100$}\label{subfig: ell6h100}
\end{subfigure}
	\caption{
        Experiment of Section~\ref{sec:blowup}.
        Evolution in time of the spatial average of
        the $z$-component of $\mm_{hk}$ for different time-step sizes.
    }\label{fig: z-evolution-blowup}
\end{figure}
In Figure~\ref{fig: z-evolution-blowup}, we fix $\ell\in\{5,6\}$, and a
time-step size of $k=h_\ell/10$
(resp. $k=h_\ell/100$) and plot the
evolution in time of the spatial average of
the $z$-component of $\mm$, i.e.,
$\langle \mm_z(t) \rangle \coloneq 1/|\Omega|\int_\Omega\mm_{hk}(t)\cdot \ee_3\,\mathrm{d}\xx$.
We see that all of the schemes capture the same evolution in time.
For $k=h_\ell/10$, the first-order pfTPS1 is more negative than 
the other plots (see Figure~\ref{subfig: ell5h10} and \ref{subfig: ell6h10}) 
due to the first-order unit length conservation,
and with a minimal variation
resulting from different values of $\beta$. Clearly the result improves for
the finer mesh size $h_6$.
On the other hand, choosing a smaller
time-step size $k=h_\ell/100$, we have a much more accurate result for both mesh sizes 
and a complete accordance of all methods for the finer mesh size $h_6$.
\begin{figure}[ht!]
    \centering
    \begin{subfigure}{0.47\textwidth}
        \centering
        \includegraphics[width=\linewidth]{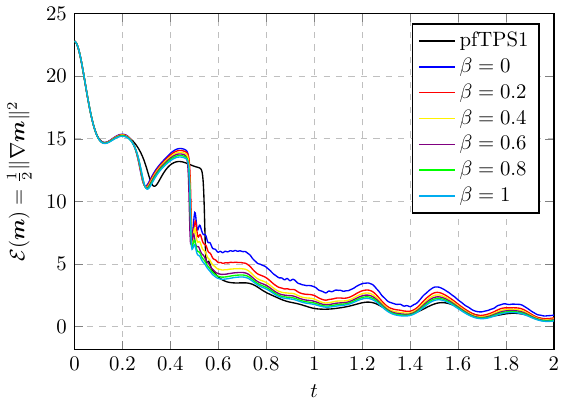}
        \vspace*{-8mm}
        \caption{Exchange energy, $k=h_6/10$}
    \end{subfigure}\hfill
    \begin{subfigure}{0.47\textwidth}
        \centering
        \includegraphics[width=\linewidth]{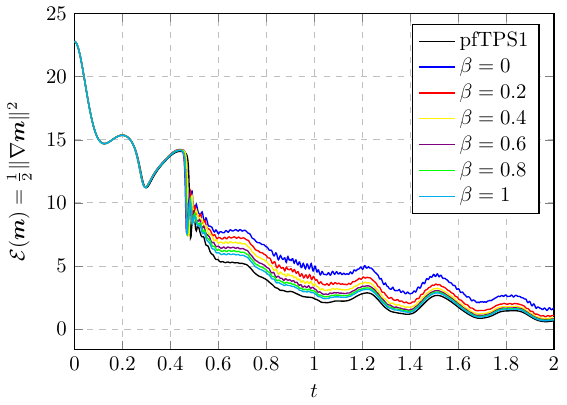}
        \vspace*{-8mm}
        \caption{Exchange energy, $k=h_6/100$}
    \end{subfigure}\hfill
    \vspace{0.35cm}
    \begin{subfigure}{0.47\textwidth}
        \centering
         \includegraphics[width=\linewidth]{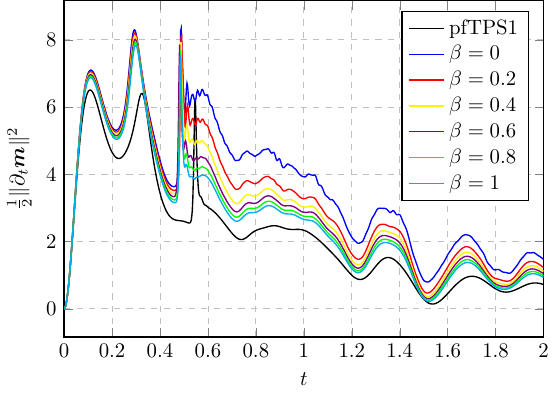}
         \vspace*{-8mm}
        \caption{Inertial energy, $k=h_6/10$}
    \end{subfigure}\hfill
    \begin{subfigure}{0.47\textwidth}
        \centering
         \includegraphics[width=\linewidth]{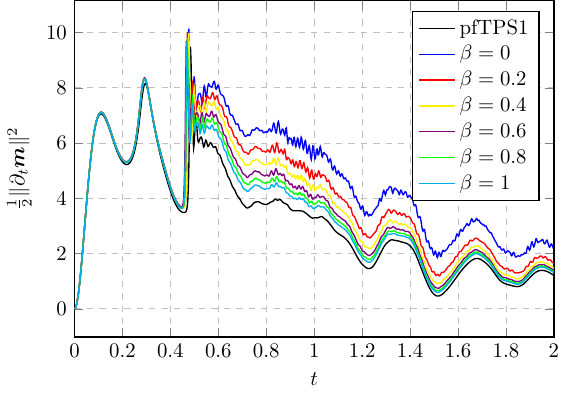}
         \vspace*{-8mm}
        \caption{Inertial energy, $k=h_6/100$}
    \end{subfigure}\hfill
    \vspace{0.35cm}
    \begin{subfigure}{0.47\textwidth}
        \centering
        \includegraphics[width=\linewidth]{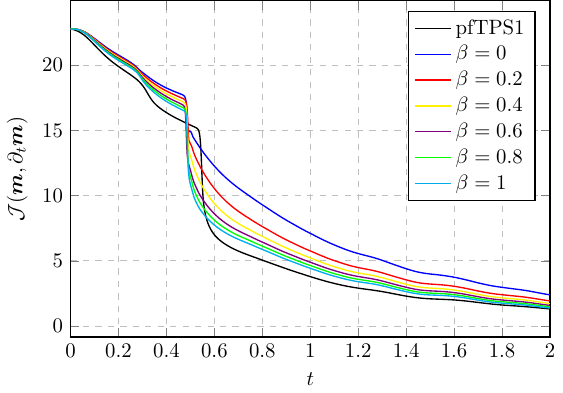}
        \vspace*{-8mm}
        \caption{Total $\mathcal{J}[\mm_{hk},\vv_{hk}^-]$, $k=h_6/10$}
    \end{subfigure}\hfill   
    \begin{subfigure}{0.47\textwidth}
        \centering
        \includegraphics[width=\linewidth]{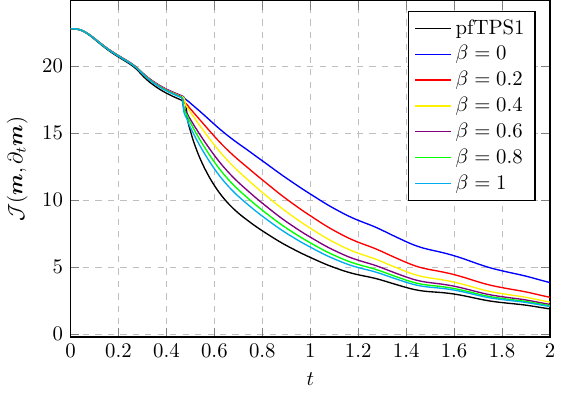}
        \vspace*{-8mm}
        \caption{Total $\mathcal{J}[\mm_{hk},\vv_{hk}^-]$, $k=h_6/100$}
    \end{subfigure}\hfill
    \vspace{0.35cm}
    \begin{subfigure}{0.33\textwidth}
        \centering
        \includegraphics[width=\linewidth]{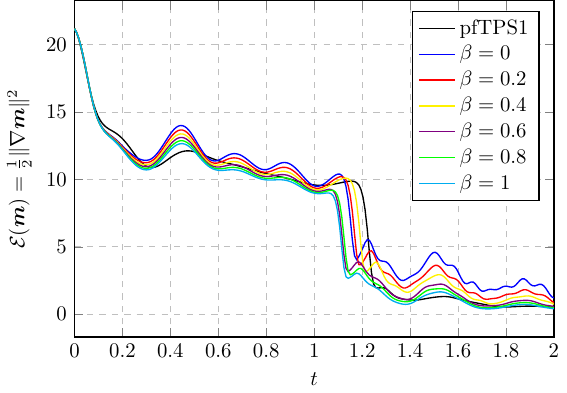}
        \caption{Exchange, $k=h_4/10$}
    \end{subfigure}\hfill
    \begin{subfigure}{0.33\textwidth}
        \centering
        \includegraphics[width=\linewidth]{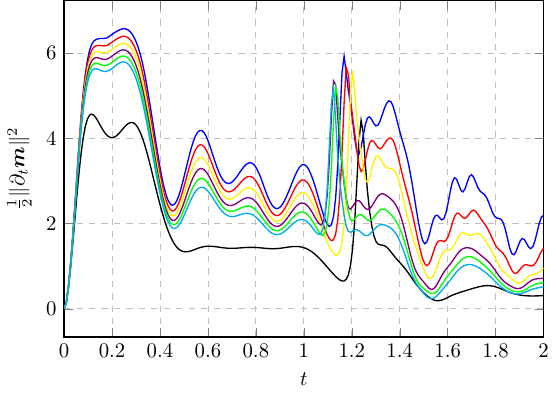}
        \caption{Inertial, $k=h_4/10$}
    \end{subfigure}\hfill
    \begin{subfigure}{0.33\textwidth}
        \centering
         \includegraphics[width=\linewidth]{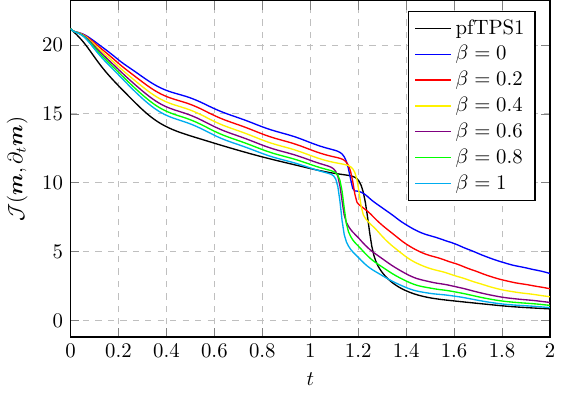}
        \caption{Total, $k=h_4/10$}
    \end{subfigure}\hfill
    \caption{Evolution in time of exchange,
    inertial and total energy.}\label{fig: energy-evolution-blowup}
\end{figure}

In Figure~\ref{fig: energy-evolution-blowup}, we
plot the evolution in time of the exchange
energy $\norm{\Grad \mm_{hk}(t)}^2/2$,
the inertial energy $\tau \norm{\vv_{hk}^-(t)}^2/2$, and
also the total energy $\mathcal{J}[\mm_{hk}(t),\vv_{hk}^-(t)]$,
for the fixed mesh size $h_6\approx0.0221$ and time-step sizes $k_1=h_6/10$ and $k_2=h_6/100$.
As in the previous experiment, we see
that values of $\beta$ close to $0$ are better
at conserving the inertial energy
than values of $\beta$ close to $1$. 
We also notice that a finer time-step size 
captures more frequent oscillations of the inertial and exchange energies.
In the last line, we plot the same energies for the coarser mesh size $h_4\approx0.0884$ 
and notice that the decay of the exchange and total energy happens later in time.

\begin{figure}[ht!]
    \centering
    \begin{subfigure}[t]{0.49\textwidth}
        \centering
        \includegraphics[width=\linewidth]{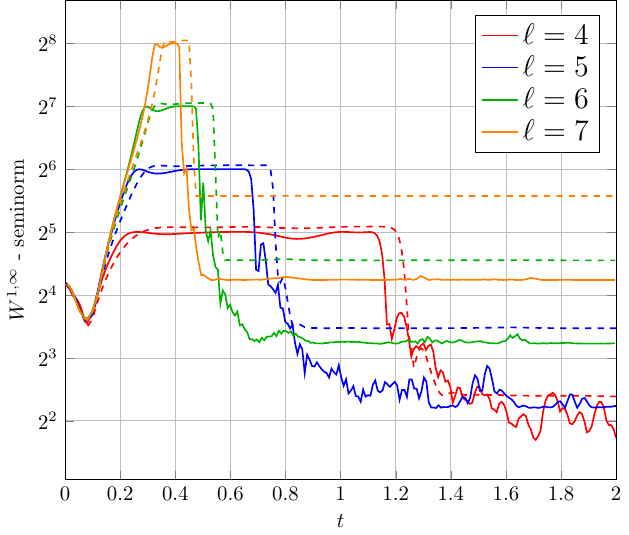}
        \caption{Evolution in time of $\norm[\bL^\infty(\Omega)]{\Grad\mm_{hk}(t)}$.
        pfTPS1 is represented with the dashed line
        and the $\beta$-scheme with $\beta=0$ (i.e., full BDF2 scheme)
        with the continuous line.\label{fig: grad-evolution-blowup}}
    \end{subfigure}\hfill
    \begin{subfigure}[t]{0.49\textwidth}
        \centering
        \includegraphics[width=\linewidth]{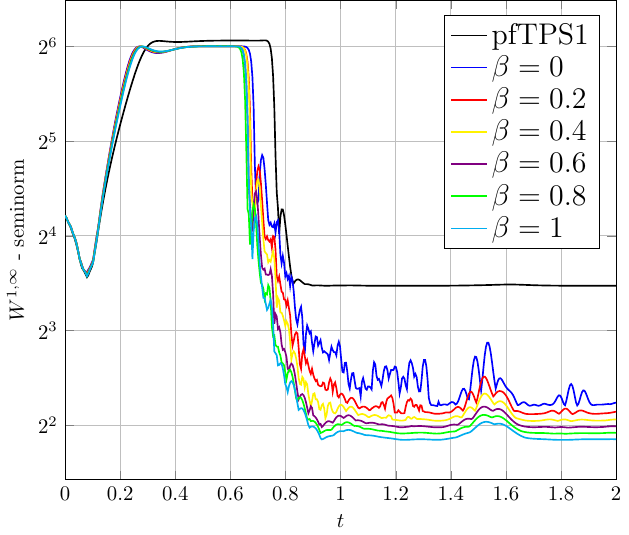}
        \caption{Evolution in time of $\norm[\bL^\infty(\Omega)]{\Grad\mm_{hk}(t)}$ 
        for different values of $\beta$ for $\ell = 5$.\label{fig: grad-evolution-blowup-beta}} 
    \end{subfigure}
    \caption{Blow up of the $\bW^{1,\infty}$-seminorm.}
\end{figure}

To conclude, Figure~\ref{fig: grad-evolution-blowup} plots
the evolution in time of
the $\bW^{1,\infty}$-seminorm $\norm[\bL^\infty(\Omega)]{\Grad\mm_{hk}(t)}$ for different mesh sizes 
$h_\ell=\sqrt{2}\cdot 2^{-\ell},$ $\ell\in\{4,5,6,7\}$ to
investigate the occurrence of a finite-time blow-up
and its dependence on the mesh size resolution.
In order to avoid overloading the plots, we compare pfTPS1 with the full
BDF2 method (i.e., we consider only $\beta=0$) for a
fixed time-step size $k = h_\ell/10$.
We see that at $t\approx 0.3$ the gradient norm reaches values
close to $2^{\ell+1}$, as observed in~\cite{ruggeri2022numerical}, where it was pointed out that 
\[
    \max_{\pphi_h\in\boldsymbol{\mathcal{M}}_{h,0}}\norm[\bL^\infty(\Omega)]{\Grad\pphi_h}
    = \max_{\pphi_h\in\boldsymbol{\mathcal{M}}_{h,0}}\max_{K\in \mathcal{T}_{h,\ell}}|\Grad\pphi_h|_K|
    = 2/2^{-\ell} = 2^{\ell+1}.
\]
This value can be exceeded here, since our schemes do not
enforce unit length at the nodes.
The value displayed above is obtained when the magnetization at the origin
$(0,0)$ points in an opposite direction with
respect to the magnetization at the closest nodes.
This configuration lasts for some
time (see the ``plateau'' in the plots), until the
magnetization at the origin flips and aligns
with the magnetization at the closest nodes,
thus decreasing the $\bW^{1,\infty}$-seminorm.
We notice that the time at which the blow-up 
seems to occur is
independent of the mesh size, while the duration of the
configuration in which the magnetization at the origin
is opposite to the one of the surrounding
nodes (i.e., the length of the plateau)
decreases with the mesh size.
This is in agreement with the numerical
experiments presented in~\cite{ruggeri2022numerical}.

Finally, in Figure~\ref{fig: grad-evolution-blowup-beta}, we fix $\ell=5$ to have a mesh size $h_5 = 0.0442$ and 
compare the evolution in time of
the $\bW^{1,\infty}$-seminorm $\norm[\bL^\infty(\Omega)]{\Grad\mm_{hk}(t)}$ 
for different values of $\beta~\in~\{0,0.2,0.4,0.6,0.8,1\}$ and $k=h_5/10$.

\subsection{Self-consistency}\label{sec:selfconsistency}
We consider the self-consistency of Algorithm~\ref{algorithm},
in comparison with the tangent plane scheme of~\cite{ruggeri2022numerical}.
We modify the self-consistency test of~\cite{albuquerque2001self}
for the LLG setting to the iLLG case.
In~\cite{albuquerque2001self}, the authors consider the discrete version
of the energy law for LLG and define
an effective Gilbert damping parameter \(\alpha_{\text{dyn}}^j\)
for all \(j \ge 0\), given by
\[
    \alpha_{\text{dyn}}^j \coloneqq -\frac{\E[\mm^{j+1}] - \E[\mm^{j}]}{k\norm{\vv^{j+1}}^2},
\]
where \(\E\) is the micromagnetic energy functional.
The expectation is that if the discrete scheme
approximates the continuous system well, then given the continuous energy law 
\[
    \frac{\mathrm{d}}{\mathrm{d} t}\E[\mm(t)] = -\alpha \norm{\mmt(t)}^2,
\]
we expect that \(\alpha_{\text{dyn}}^j\) should be close to the
Gilbert damping parameter \(\alpha\).
In the iLLG setting, we consider the discrete version
of the energy law~\eqref{eq:strong_energy_law}, and define
\[
    \alpha_{\text{dyn}}^j
        = -\frac{\mathcal{J}[\mm^{j+1},\vv^{j+1}] - \mathcal{J}[\mm^{j},\vv^j]}{k\norm{\vv^{j+1}}^2}.
\]
We first make two observations.
As we are only dealing with weak solutions
in the sense of Definition~\ref{def: weak solution},
the energy law~\eqref{eq:strong_energy_law}
is expected to hold in the sense of an
inequality integrated over the full
time \((0,T)\) (see Definition~\ref{def: weak solution}\eqref{def: weak solution iv}).
It is therefore too much to expect that \(\alpha_{\text{dyn}}^{j} \to \alpha\)
in a pointwise sense for each \(j\in\N\).
We generally do expect
however that \(\alpha_{\text{dyn}}^{j} \geq \alpha\)
for sufficiently small \(k\).

Moreover, we apply an implicit scheme along the lines of~\cite{ruggeri2022numerical}
for the initialization step. Setting \(\vv_h^0 \equiv \0\),
we can deduce from~\eqref{eq:initial_energy_law}
that
\[
    \mathcal{J}[\mm_h^1,\vv_h^1]
    + \alpha k\norm{\vv_h^1}^2
    + \frac{\tau}{2}\norm{\vv_h^1}^2
    \leq  \mathcal{J}[\mm_h^0,\vv_h^0].
\]
It follows that we have
\[
    \alpha_{\text{dyn}}^0
    = -\frac{\mathcal{J}[\mm_h^1,\vv_h^1] - \mathcal{J}[\mm_h^0,\vv_h^0]}{k\norm{\vv_h^1}^2}
    \geq \alpha + \frac{\tau}{2k}.
\]
Therefore, as the time-step size \(k\to 0\),
we expect that \(\alpha_{\text{dyn}}^0\) diverges to \(+\infty\).
This is expected given that the Gilbert damping term is not the only
term responsible for the energy dissipation at the initial step.
For \(j>0\), the exact form of \(\alpha_{\text{dyn}}^j\) is more complicated,
and found from~\eqref{eq:discrete_energy_equality}.
For all values of \(\beta\in[0,1]\),~\eqref{eq:discrete_energy_equality}
contains unsigned terms and so there is no guarantee that
\(\alpha_{\text{dyn}}^j>\alpha\).
For the experiments, we set \(\Omega = (0,1)^3\), \(\alpha = 0.1\), \(\tau = 0.1\),
and \(T = 1/2, k =\num{3.125e-4}\). We include a Zeeman field
of \(\ff = (1,0,0)\).
The initial condition is given by~\eqref{eq: high_energy_initial_condition}.
The mesh consists of \(154\) elements
and \(63\) nodes, and satisfies \(h_{\max} \approx 0.64\).
As mentioned above, we should expect that
\(\alpha_{\text{dyn}}^0 \geq \alpha + \tau/(2k) = 160.1\).

\begin{figure}[ht!]
    \begin{subfigure}{0.49\textwidth}
    \includegraphics{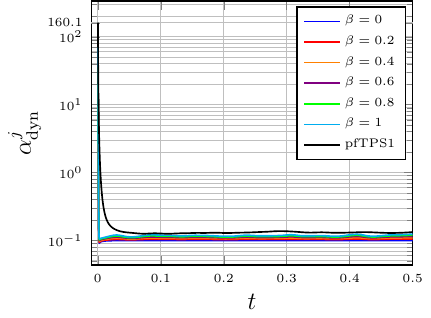}
    \caption{\label{fig:selfconsistency}}
    \end{subfigure}
    \hfill
    \begin{subfigure}{0.49\textwidth}
    \includegraphics{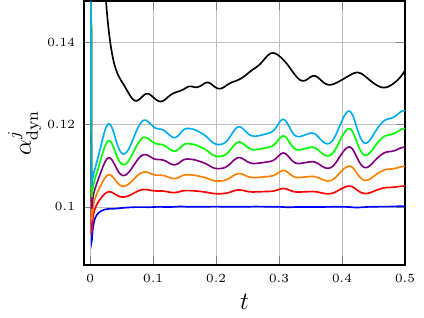}
    \caption{\label{fig:selfconsistency_two}}
    \end{subfigure}
    \caption{Experiment of Section~\ref{sec:selfconsistency}.
    (a) Evolution in time of the effective damping parameter
    \(\alpha_{\text{dyn}}^j\). (b) Restricted \(y\)-axis version of (a).}
\end{figure}

We display the result in Figure~\ref{fig:selfconsistency},
and the same plot where the vertical axis is cut off at 0.15 is shown in
Figure~\ref{fig:selfconsistency_two}.
Firstly, we see in Figure~\ref{fig:selfconsistency} that the initialization step
performs as expected, hovering slightly above \(160.1\) for all methods.
In Figure~\ref{fig:selfconsistency_two} however, we clearly see that
pfTPS1 underperforms, with significantly larger damping overall
and a much more gradual approach to \(\alpha = 0.1\), taking
until about \(t=0.05\) to stop decreasing.
On the other hand, we see that for \(\beta=0\), the effective Gilbert parameter is quickly
attained after initialization, and does indeed dip slightly below \(\alpha\).
As we increase \(\beta\), we see the effective Gilbert parameter
becoming larger, but still remaining below the one corresponding to pfTPS1.

\FloatBarrier
\section{Proofs}\label{sec:proofs}

In this section,
we collect the proofs of all results of the work.

\subsection{Well-posedness, discrete energy law, and stability}
We first show some nodal properties satisfied by the approximations,
which guarantee the well-posedness of the discrete tangent space
for each step of Algorithm~\ref{algorithm}.

\begin{lemma}\label{lem: well-posedness discrete tangent space}
	For every node $\zz\in\mathcal{N}_h$ it holds that
	\begin{itemize}
		\item[\rm (a)] $|\mm_h^{j+1}(\zz)|^2 = \frac{4}{3} |\mm_h^{j}(\zz)|^2 -\frac{1}{3} |\mm_h^{j-1}(\zz)|^2 + k^4 |{\rm d}_t^2\mm_h^{j+1}(\zz)|^2$ for every $j \ge 1$;
		\item[\rm (b)] $|\mm_h^{j+1}(\zz)|^2 \geq |\mm_h^{j}(\zz)|^2 \geq 1$ for every $j \ge 0$;
		\item[\rm (c)] $|\mhat_h^{j+1}(\zz)| \geq 1$ for every $j \ge 0$.
	\end{itemize}
	This implies that the discrete tangent space $\bT_h(\mhat_h^{j+1})$ for all $j \ge 0$ is well-defined.
\end{lemma}

\begin{proof}
Let $\zz\in\mathcal{N}_h$ and $j\ge1$.
By Algorithm~\ref{algorithm}, it holds that $\mhat_h^{j+1}=2\mm_h^{j}-\mm_h^{j-1}$ and $3\mm_h^{j+1}=4\mm_h^{j}-\mm_h^{j-1}+2k\vv_h^{j+1}$.
Since $k^2\dt^2\mm_h^{j+1}=\mm_h^{j+1}-2\mm_h^{j}+\mm_h^{j-1}$, this yields
\[
    \mm_h^{j+1}=\mhat_h^{j+1}+k^2\dt^2\mm_h^{j+1}
    \quad\text{and}\quad
    2k^2\dt^2\mm_h^{j+1}=\frac{4}{3}k\vv_h^{j+1}-\frac{4}{3}\big(\mm_h^{j}-\mm_h^{j-1}\big).
\]
Expanding the square in the first identity, inserting the second one, and using that $\vv_h^{j+1}(\zz)\cdot\mhat_h^{j+1}(\zz)=0$ (since $\vv_h^{j+1}\in\bT_h(\mhat_h^{j+1})$) as well as $\mhat_h^{j+1}=2\mm_h^{j}-\mm_h^{j-1}$, we obtain
\begin{align*}
    |\mm_h^{j+1}(\zz)|^2-k^4|\dt^2\mm_h^{j+1}(\zz)|^2
    &=|\mhat_h^{j+1}(\zz)|^2+2k^2\dt^2\mm_h^{j+1}(\zz)\cdot\mhat_h^{j+1}(\zz)\\
    &=\Big(\mhat_h^{j+1}(\zz)-\frac{4}{3}\big(\mm_h^{j}(\zz)-\mm_h^{j-1}(\zz)\big)\Big)\cdot\mhat_h^{j+1}(\zz)\\
    &=\Big(\frac23\mm_h^{j}(\zz)+\frac13\mm_h^{j-1}(\zz)\Big)\cdot\big(2\mm_h^{j}(\zz)-\mm_h^{j-1}(\zz)\big)\\
    &=\frac{4}{3}|\mm_h^{j}(\zz)|^2-\frac{1}{3}|\mm_h^{j-1}(\zz)|^2.
\end{align*}
This proves~\textrm{(a)}. An immediate consequence of \textrm{(a)} is that
\begin{align*}
	|\mm_h^{j+1}(\zz)|^2 &\geq |\mm_h^{j}(\zz)|^2 + \frac{1}{3} \big(|\mm_h^{j}(\zz)|^2 -|\mm_h^{j-1}(\zz)|^2\big).
\end{align*}
Since $|\mm_h^1(\zz)|^2 = |\mm_h^0(\zz)|^2 + k^2 |\vv_h^1(\zz)|^2 \ge |\mm_h^0(\zz)|^2 = 1$, induction on $j$ proves~\textrm{(b)}.
Finally, for $j=0$, (c) follows from $\mhat_h^1=\mm_h^0$ and $|\mm_h^0(\zz)|=1$.
For $j\ge1$, the reverse triangle inequality and~\textrm{(b)} give
\[
    |\mhat_h^{j+1}(\zz)|=|2\mm_h^{j}(\zz)-\mm_h^{j-1}(\zz)|
    \ge 2|\mm_h^{j}(\zz)|-|\mm_h^{j-1}(\zz)|
    \ge|\mm_h^{j}(\zz)|\ge1.
\]
This concludes the proof of~\textrm{(c)}.
\end{proof}

Next, we prove the well-posedness of the algorithm and derive the discrete energy law
satisfied by the approximations.

\begin{proof}[Proof of Proposition~\ref{prop: energy-estimate-1}]
For the initialization step, we have the linear system
\begin{equation*}
    \tau\inner{\dt\vv_h^1}{\pphi_h}
    +\alpha \inner{\vv_h^1}{\pphi_{h}}
    + \inner{\mm_{h}^{0}\times \vv_h^1}{\pphi_{h}}
    +  k\inner{\Grad\vv_h^1}{\Grad \pphi_{h}}
    = -\inner{\Grad\mm_{h}^{0}}{\Grad \pphi_{h}}.
\end{equation*}
This is easily seen to have
a unique solution $\vv_h^1 \in \bT_h(\mm_h^{0})$ via the Lax--Milgram lemma.
Inserting $\pphi_h = k\vv_h^1$ in~\eqref{alg1:magnetization_update_first} yields
\begin{equation*}
    \tau\inner{\vv_h^1 - \vv_h^0}{\vv_h^1}
    +\alpha k\norm{\vv_h^1}^2
    +  k\inner{\Grad\mm_{h}^{1}}{\Grad \vv_h^1}
    = 0.
\end{equation*}
We recall the vector identity for \(\aaa,\bb\in\R^3\)
\begin{equation}\label{eq:difference_vector_identity}
    2(\aaa-\bb)\cdot\aaa = |\aaa|^2-|\bb|^2+|\aaa-\bb|^2.
\end{equation}
With $\mm_h^1 = \mm_h^0 + k \vv_h^1$
and $\vv_h^1 = \vv_h^0 + k \dt\vv_h^1$,
we use~\eqref{eq:difference_vector_identity} to see
\begin{align*}
\inner{\Grad\mm_{h}^{1}}{ k \Grad \vv_h^1}
&= \frac{1}{2}\norm{\Grad\mm_h^1}^2
    -\frac{1}{2}\norm{\Grad\mm_h^0}^2
    +\frac{k^2}{2}\norm{\Grad\vv_h^1}^2,\\
    \label{eq: vector identity}
    \inner{\vv_h^1 - \vv_h^0}{\vv_h^1}
        &= \frac{1}{2}\norm{\vv_h^1}^2
        - \frac{1}{2}\norm{\vv_h^0}^2
        + \frac{k^2}{2}\norm{\dt\vv_h^1}^2.
\end{align*}
Hence, we have
\begin{equation*}
    \mathcal{J}[\mm_h^1,\vv_h^1]
    + \alpha k\norm{\vv_h^1}^2
    + \frac{\tau k^2}{2}\norm{\dt\vv_h^1}^2
    +  \frac{k^2}{2}\norm{\Grad\vv_h^1}^2
    =  \mathcal{J}[\mm_h^0,\vv_h^0],
\end{equation*}
which is~\eqref{eq:initial_energy_law}.
Rearranging~\eqref{alg1:magnetization_update}, we have
\begin{multline*}
    \beta\frac{\tau}{k}\inner{\vv_h^{j+1}}{\pphi_h}
    +(1-\beta) \frac{3\tau}{2k}\inner{\vv_h^{j+1}}{\pphi_h}
    +\alpha \inner{\vv_{h}^{j+1}}{\pphi_{h}}\\
    + \inner{\mhat_{h}^{j+1}\times \vv_{h}^{j+1}}{\pphi_{h}}
    + \frac{2k}{3}\inner{\Grad\vv_h^{j+1}}{\Grad \pphi_{h}}\\
    \mkern-100mu=\beta\frac{\tau}{k}\inner{\vv_h^j}{\pphi_h}
    +(1-\beta)\frac{\tau}{2k}\inner{4 \vv_h^j-\vv_h^{j-1}}{\pphi_h}
    -\frac{1}{3}\inner{\Grad(4\mm_h^j-\mm_h^{j-1})}{\Grad \pphi_{h}}.
\end{multline*}
The left-hand side is clearly coercive in $\bH^1(\Omega)$
with the right-hand side being bounded.
It follows that the solution $\vv_h^{j+1} \in \bT_h(\mhat_h^{j+1})$ exists
and is unique.
Inserting $\pphi_h = k\vv_h^{j+1}$ yields that
\begin{align*}
    \beta\tau\inner{\vv_h^{j+1}-\vv_h^j}{\vv_h^{j+1}}
    + (1-\beta) \frac{\tau}{2}\inner{3\vv_h^{j+1}-4\vv_h^j+\vv_h^{j-1}}{\vv_h^{j+1}}
    + \alpha k\norm{\vv_{h}^{j+1}}^2\\
    +\frac{k}{3}\inner{\Grad(4\mm_h^j-\mm_h^{j-1}+2k\vv_h^{j+1})}{\Grad \vv_{h}^{j+1}}=0.
\end{align*}
Define $\MM_h^{j} \coloneqq (\mm_h^j,\mm_h^{j-1})$ and $\VV_h^j \coloneqq (\vv_h^j, \vv_h^{j-1})$, as well as the notation $\norm[G]{\mathcal{U}}^2 \coloneqq \int_\Omega |\mathcal{U}|_G^2\,\mathrm{d}\xx$ for pairs $\mathcal{U}$ of (vector- or matrix-valued) functions, where $|\cdot|_G$ is defined in~\eqref{eq:G-norm}.
Applying~\eqref{eq:G_inner_product} yields
\begin{align*}
    \frac{1}{2}\inner{3\vv_h^{j+1}-4\vv_h^{j}+\vv_h^{j-1}}{\vv_h^{j+1}}
    = \norm[G]{\VV_h^{j+1}}^2-\norm[G]{\VV_h^{j}}^2
    + \frac{k^4}{4}\norm{\dt^2\vv_h^{j+1}}^2.
\end{align*}
Analogously, we obtain that
\begin{align*}
    \frac{k}{3}\inner{\Grad(4\mm_h^j-\mm_h^{j-1}+2k\vv_h^{j+1})}{\Grad \vv_{h}^{j+1}}
    &= k\Big\langle\Grad \mm_h^{j+1},\Grad\frac{1}{2k}\big( 3 \mm_h^{j+1} - 4 \mm_h^j + \mm_h^{j-1}\big)\Big\rangle \\ 
    & = \norm[G]{\Grad\MM_h^{j+1}}^2
    - \norm[G]{\Grad\MM_h^{j}}^2
    +\frac{k^4}{4}\norm{\dt^2\Grad\mm_h^{j+1}}^2.
\end{align*}
Putting everything together and applying again the vector
identity~\eqref{eq:difference_vector_identity} yields that
\begin{equation}
    \begin{aligned}\label{eq:loop_energy_law}
   \beta \tau & \Big( \frac{1}{2} \norm{\vv_h^{j+1}}^2 - \frac{1}{2} \norm{\vv_h^{j}}^2 
    + \frac{k^2}{2} \norm{\dt\vv_h^{j+1}}^2 \Big)\\
    &+ (1-\beta)  \tau \Big( \norm[G]{\VV_h^{j+1}}^2 - \norm[G]{\VV_h^{j}}^2 + \frac{k^4}{4}\norm{\dt^2\vv_h^{j+1}}^2 \Big)\\
    &+\alpha  k\norm{\vv_{h}^{j+1}}^2
    + \norm[G]{\Grad\MM_h^{j+1}}^2
    - \norm[G]{\Grad\MM_h^{j}}^2
    + \frac{k^4}{4}\norm{\dt^2\Grad\mm_h^{j+1}}^2
    =0.
\end{aligned}
\end{equation}
Using~\eqref{eq:G_norm_difference}, we see that for $\ell \in \{j,j+1\}$ it holds that
\begin{equation}\label{eq: Grad M}
    \norm[G]{\Grad\MM_h^{\ell}}^2
        = \frac{1}{2} \norm{\Grad\mm_h^{\ell}}^2
        + \frac{k}{4}\dt\norm{\Grad\mm_h^{\ell}}^2
        + \frac{k^2}{2} \norm{\dt\Grad\mm_h^{\ell}}^2
\end{equation}
and
\begin{equation}\label{eq: V}
    \norm[G]{\VV_h^{\ell}}^2
    = \frac{1}{2} \norm{\vv_h^{\ell}}^2
    + \frac{k}{4}\dt\norm{\vv_h^{\ell}}^2
    + \frac{k^2}{2} \norm{\dt\vv_h^{\ell}}^2.
\end{equation}
Inserting these relations into~\eqref{eq:loop_energy_law},
we obtain exactly~\eqref{eq:discrete_energy_equality}.

To prove~\eqref{eq:corollary_energy_law}, it suffices to sum~\eqref{eq:discrete_energy_equality} over $j=1,\ldots,n-1$. 
Thanks to telescoping sums, this yields that
\begin{equation*}
    \begin{split}
        \mathcal{J}&[\mm_h^n,\vv_h^n]+\alpha k \sum_{j=1}^{n-1}\norm{\vv_h^{j+1}}^2+\frac{\tau k^2}{2} \sum_{j=1}^{n-1}\norm{\dt \vv_h^{j+1}}^2
        +\frac{k^2}{2}\norm{\dt \Grad\mm_h^n}^2\\
        &+\frac{k^4}{4}\sum_{j=1}^{n-1}\norm{\dt^2\Grad\mm_h^{j+1}}^2+\frac{k}{4}\dt\norm{\Grad\mm_h^n}^2+(1-\beta)\tau\Big[\frac{ k^4}{4}\sum_{j=1}^{n-1}\norm{\dt^2\vv_h^{j+1}}^2+\frac{k}{4} \dt\norm{\vv_h^n}^2\Big]\\
        &=\mathcal{J}[\mm_h^1,\vv_h^1] + \frac{k^2}{2}\norm{\dt\Grad\mm_h^1}^2+\frac{k}{4}\dt\norm{\Grad\mm_h^1}^2+(1-\beta)\tau\Big[\frac{k^2}{2}\sum_{j=1}^{n-1}\norm{\dt\vv_h^j}^2+\frac{k}{4}\dt\norm{\vv_h^1}^2\Big].
    \end{split}
\end{equation*}
Inserting the energy law~\eqref{eq:initial_energy_law} for the first time step ($j=0$)
in the latter,
and noting that $\dt\mm_h^1 = \vv_h^1$ allows us to write
\begin{equation*}\label{eq: energy_equality_last_step}
    \begin{split}
        \mathcal{J}&[\mm_h^n,\vv_h^n]+\alpha k \sum_{j=0}^{n-1}\norm{\vv_h^{j+1}}^2+\frac{k}{4}\dt\norm{\Grad\mm_h^n}^2+\frac{\tau k^2}{2} \sum_{j=0}^{n-1}\norm{\dt \vv_h^{j+1}}^2
        +\frac{k^2}{2}\norm{\dt \Grad\mm_h^n}^2\\
        &+\frac{k^4}{4}\sum_{j=1}^{n-1}\norm{\dt^2\Grad\mm_h^{j+1}}^2+(1-\beta)\tau\Big[\frac{ k^4}{4}\sum_{j=1}^{n-1}\norm{\dt^2\vv_h^{j+1}}^2+\frac{k}{4} \dt\norm{\vv_h^n}^2\Big]\\
        &=\mathcal{J}[\mm_h^0,\vv_h^0] + \frac{k}{4}\dt\norm{\Grad\mm_h^1}^2+(1-\beta)\tau\Big[\frac{k^2}{2}\sum_{j=1}^{n-1}\norm{\dt\vv_h^j}^2+\frac{k}{4}\dt\norm{\vv_h^1}^2\Big].
    \end{split}
\end{equation*}
Exploiting the equality
\begin{align*}
\sum_{j=0}^{n-1}\norm{\dt\vv_h^{j+1}}^2 - (1-\beta)\sum_{j=1}^{n-1}\norm{\dt\vv_h^{j}}^2 & = \norm{\dt\vv_h^{n}}^2+\beta \sum_{j=1}^{n-1}\norm{\dt\vv_h^j}^2\\
&=\beta \sum_{j=0}^{n-1}\norm{\dt\vv_h^{j+1}}^2+(1-\beta)\norm{\dt\vv_h^n}^2,
\end{align*}
we conclude the proof of~\eqref{eq:corollary_energy_law}.
\end{proof}

\begin{remark}
    Notice that if $\beta = 0$ we would
    lose the boundedness of $k^2\sum_{j=0}^{n-1}\norm{\dt\vv_h^{j+1}}^2$
    and consequently the boundedness
    of $k^2\sum_{j=1}^{n-1}\norm{\dt^2\mm_h^{j+1}}^2$
    (see Lemma~\ref{prop:dtt_stability} below),
    which is (analytically, but not numerically)
    crucial for the second-order in time unit length
    constraint and to pass to the limit in the variational
    formulation in the proof of Theorem~\ref{thm: main result}.  
    The case $\beta = 1$ is excellent analytically,
    since it guarantees the needed boundedness of
    $k^2\sum_{j=1}^{n-1}\norm{\dt^2\mm_h^{j+1}}^2$, but
    $\beta=0$ (i.e., BDF2 discretization for the inertial term) performs
    better numerically, especially in the case where the
    time-step size is large in terms of energy conservation
    (this is a problem noticed in~\cite{ruggeri2022numerical},
    and it is why the author recommends using the nonlinear angular momentum
    method instead of the tangent plane scheme).
\end{remark}

The following lemma,
adapted from~\cite[Lemma~4.5]{dong2024bdf},
provides us with useful bounds and allows us to prove
unconditional second-order consistency for the constraint when $\beta \in (0,1]$.

\begin{lemma}\label{prop:dtt_stability}
For all $j\geq 1$, it holds that
\begin{equation}\label{eq: bound dt m j+1}
\norm{\dt\mm_h^{j+1}}
    \leq \frac{2}{3}\norm{\vv_h^{j+1}}
    + \frac{1}{3}\norm{\dt\mm_h^{j}}
\end{equation}
and, for all $n\ge 2$,
\begin{equation}\label{eq: bound dt2 m j+1}
    \sum_{j=2}^{n-1}\norm{\dt^2\mm_h^{j+1}}^2
    \leq \frac{2}{7}\norm{\dt^2 \mm_h^{2}}^2
    +\frac{8}{7}\sum_{j=2}^{n-1}\norm{\dt\vv_h^{j+1}}^2.
\end{equation}
\end{lemma}

\begin{proof}
For \(j\geq 1\) it holds that
\[
    3 \dt\mm_h^{j+1} = 2\vv_h^{j+1} + \dt\mm_h^{j}.
\]
Dividing by $3$ and using the triangle inequality, we obtain~\eqref{eq: bound dt m j+1}.
For $j\ge 2$, it holds that
\[
    3\dt^2\mm_h^{j+1} = 2\dt\vv_h^{j+1} + \dt^2\mm_h^{j}.
\]
Using this and the relation
\((2a+b)^2 \leq 8a^2 + 2b^2\),
it follows that
\[
    \norm{\dt^2\mm_h^{j+1}}^2 \leq \Big(\frac{2}{3}\norm{\dt\vv_h^{j+1}}
    + \frac{1}{3}\norm{\dt^2\mm_h^{j}}\Big)^2
    \leq \frac{8}{9}\norm{\dt\vv_h^{j+1}}^2
    + \frac{2}{9}\norm{\dt^2\mm_h^{j}}^2.
\]
Summing over \(j=2,\dots, n-1\), we see that
\begin{align*}
    \norm{\dt^2 \mm_h^n}^2
    + \frac{7}{9}\sum_{j=3}^{n-1}\norm{\dt^2\mm_h^{j}}^2
    - \frac{2}{9}\norm{\dt^2 \mm_h^{2}}^2 &= \sum_{j=2}^{n-1}\norm{\dt^2\mm_h^{j+1}}^2
    - \frac{2}{9}\sum_{j=2}^{n-1}\norm{\dt^2\mm_h^{j}}^2\\
    &\leq \frac{8}{9}\sum_{j=2}^{n-1}\norm{\dt\vv_h^{j+1}}^2.
\end{align*}
Multiplying both sides by \(\dfrac{9}{7}\) and weakening the inequality gives
\[
    \sum_{j=3}^{n}\norm{\dt^2\mm_h^{j}}^2
    \leq \frac{2}{7}\norm{\dt^2 \mm_h^{2}}^2
    + \frac{8}{7}\sum_{j=2}^{n-1}\norm{\dt\vv_h^{j+1}}^2.
\]
This concludes the proof.
\end{proof}

We conclude this subsection by proving the unconditional stability of Algorithm~\ref{algorithm},
which we also exploit to derive the explicit bounds for the error in the realization of the unit length constraint.

\begin{proof}[Proof of Proposition~\ref{prop:stability}]
Let $\beta\in[0,1]$. We split the proof into three steps.

\emph{Step~1: Proof of~\eqref{eq:stability} with $\norm{\Grad\mm_h^n}^2$ in place of $\norm[\bH^1(\Omega)]{\mm_h^n}^2$.}
Recall the definition of the energy functional~\eqref{eq: Energy functional} and the discrete energy equality~\eqref{eq:corollary_energy_law}. Notice that its right-hand side is bounded by
\begin{multline*}
\mathcal{J}[\mm_h^0,\vv_h^0]+ \frac{k}{4}\dt\norm{\Grad\mm_h^1}^2+(1-\beta)\frac{\tau k}{4}\dt\norm{\vv_h^1}^2 \\
= \frac{1}{4}\norm{\Grad\mm_h^{0}}^2 + \frac{1}{4}\norm{\Grad\mm_h^{1}}^2
+ (1+\beta)\frac{\tau}{4} \norm{\vv_h^0}^2
+ (1-\beta)\frac{\tau}{4} \norm{\vv_h^1}^2
\stackrel{\eqref{eq:initial_energy_law}}{\lesssim}
\norm{\vv_h^{0}}^2 + \norm{\Grad\mm_h^{0}}^2,
\end{multline*}
where the hidden constant depends only on $\tau$ (uniformly in $\beta$).
Recall $\MM_h^{j} \coloneqq (\mm_h^j,\mm_h^{j-1})$ and $\VV_h^j \coloneqq (\vv_h^j, \vv_h^{j-1})$, as well as the definition of the $G$-norm for (vector- or matrix-valued) functions, and define
\[
    \mathcal{J}_{G,h}^n
    \coloneqq
    \norm[G]{\Grad\MM_h^n}^2
    + \frac{\beta\tau}{2}\norm{\vv_h^n}^2
    + (1-\beta)\tau\norm[G]{\VV_h^n}^2,
    \qquad n\ge 1.
\]
By~\eqref{eq: Grad M} and~\eqref{eq: V}, and since $\mathcal{J}[\mm_h^n,\vv_h^n]=\frac{1}{2}\norm{\Grad\mm_h^n}^2+\frac{\tau}{2}\norm{\vv_h^n}^2$ by~\eqref{eq: exchange energy} and~\eqref{eq: Energy functional}, this can be rewritten as
\begin{align*}
    \mathcal{J}_{G,h}^n
    &=
    \frac{1}{2}\norm{\Grad\mm_h^n}^2
    + \frac{k}{4}\dt\norm{\Grad\mm_h^{n}}^2
    + \frac{k^2}{2}\norm{\dt\Grad\mm_h^{n}}^2
    + \frac{\tau}{2}\norm{\vv_h^n}^2
    \\
    &\qquad
    + (1-\beta)\frac{\tau k}{4}\dt\norm{\vv_h^{n}}^2
    + (1-\beta)\frac{\tau k^2}{2}\norm{\dt\vv_h^{n}}^2
    \\
    &=
    \mathcal{J}[\mm_h^n,\vv_h^n]
    + \frac{k^2}{2}\norm{\dt\Grad\mm_h^{n}}^2
    + \frac{k}{4}\dt\norm{\Grad\mm_h^{n}}^2
    + (1-\beta)\tau\Big[\frac{k}{4}\dt\norm{\vv_h^{n}}^2
    + \frac{k^2}{2}\norm{\dt\vv_h^{n}}^2\Big].
\end{align*}
Moreover, we define
\[
    \varrho_1
    \coloneqq
    \frac{k}{4}\dt\norm{\Grad\mm_h^1}^2
    + (1-\beta)\frac{\tau k}{4}\dt\norm{\vv_h^1}^2.
\]
Inserting these expressions into~\eqref{eq:corollary_energy_law}, we can rewrite the discrete energy identity as
\begin{equation}\label{eq: bounds_intermediate}
\begin{aligned}
    &\mathcal{J}_{G,h}^n
    + \alpha k\sum_{j=0}^{n-1}\norm{\vv_h^{j+1}}^2
    + \frac{\beta\tau k^2}{2}\sum_{j=0}^{n-1}\norm{\dt\vv_h^{j+1}}^2
    + \frac{(1-\beta)\tau k^4}{4}\sum_{j=1}^{n-1}\norm{\dt^2\vv_h^{j+1}}^2
    \\
    &\qquad
    + \frac{k^4}{4}\sum_{j=1}^{n-1}\norm{\dt^2\Grad\mm_h^{j+1}}^2
    = \mathcal{J}[\mm_h^0,\vv_h^0]+\varrho_1
\end{aligned}
\end{equation}
for all $n\ge2$. By the preceding bound, the right-hand side of~\eqref{eq: bounds_intermediate} is $\lesssim\norm{\vv_h^0}^2+\norm{\Grad\mm_h^0}^2$.
All terms on the left-hand side are nonnegative.
Moreover, the norm equivalence~\eqref{eq: G-norm-equivalence}
gives
\[
    \norm[G]{\Grad\MM_h^n}^2
    \ge \mu_-\norm{\Grad\mm_h^n}^2
\]
and, since $0<\mu_-<1/2$,
\begin{align*}
    \frac{\beta\tau}{2}\norm{\vv_h^n}^2
    +(1-\beta)\tau\norm[G]{\VV_h^n}^2
    &\ge
    \tau\Big(\frac{\beta}{2}+(1-\beta)\mu_-\Big)
    \norm{\vv_h^n}^2
    \ge \tau\mu_-\norm{\vv_h^n}^2.
\end{align*}
Since $\alpha,\tau>0$, these estimates yield
\begin{equation}\label{eq:first_bound}
    \begin{split}
         \norm{\Grad\mm_h^{n}}^2&+\norm{\vv_h^{n}}^2  
        +  k\sum_{j=0}^{n-1}\norm{\vv_{h}^{j+1}}^2
        + \beta k^2 \sum_{j=0}^{n-1}\norm{\dt\vv_h^{j+1}}^2
        + (1-\beta)k^4\sum_{j=1}^{n-1}\norm{\dt^2\vv_h^{j+1}}^2\\
        &+ k^4\sum_{j=1}^{n-1}\norm{\dt^2\Grad\mm_h^{j+1}}^2
        \lesssim \norm{\vv_h^{0}}^2
        + \norm{\Grad\mm_h^{0}}^2 \lesssim 1,
    \end{split}
\end{equation}
where the hidden constant depends only on $\alpha$ and $\tau$.
Moreover, in the case $\beta>0$, Lemma~\ref{prop:dtt_stability} and~\eqref{eq:first_bound} yield
\begin{align*}
    k^2\sum_{j=2}^{n-1}\norm{\dt^2\mm_h^{j+1}}^2
    &\stackrel{\eqref{eq: bound dt2 m j+1}}{\lesssim} k^2\norm{\dt^2\mm_h^{2}}^2
    + k^2\sum_{j=2}^{n-1}\norm{\dt\vv_h^{j+1}}^2
    \stackrel{\eqref{eq:first_bound}}{\lesssim} k^2\norm{\dt^2\mm_h^{2}}^2 + \frac{1}{\beta}.
\end{align*}
The term \(k^2\norm{\dt^2\mm_h^{2}}^2\) is also bounded, uniformly in $\beta$. Indeed,
\begin{align*}
    k^2 \dt^2\mm_h^{2} &= \mm_h^2 - 2\mm_h^1 + \mm_h^0
    = \frac{2}{3}k(\vv_h^2 - \dt\mm_h^1)
    = \frac{2}{3}k(\vv_h^2 - \vv_h^1)
\end{align*}
and therefore 
\begin{align*}
    k^4\norm{\dt^2\mm_h^{2}}^2
    &\lesssim k^2\big(\norm{\vv_h^1}^2
    +  \norm{\vv_h^2}^2\big)
    \mathrel{\stackrel{\eqref{eq:initial_energy_law},\eqref{eq:first_bound}}{\lesssim}} k^2.
\end{align*}
Altogether, we are led to
\begin{equation}\label{eq:second_bound}
    k^2\sum_{j=1}^{n-1}\norm{\dt^2\mm_h^{j+1}}^2\lesssim 1+\frac{1}{\beta},
\end{equation}
where the hidden constant depends only on $\alpha$ and $\tau$.
Multiplying~\eqref{eq:second_bound} by $\beta\le1$, we obtain $\beta k^2\sum_{j=1}^{n-1}\norm{\dt^2\mm_h^{j+1}}^2\lesssim\beta+1\le2$, i.e., the factor $\beta$ in~\eqref{eq:stability} compensates for the factor $\beta^{-1}$ in~\eqref{eq:second_bound}. Together with~\eqref{eq:first_bound}, this proves~\eqref{eq:stability} with $\norm{\Grad\mm_h^n}^2$ in place of $\norm[\bH^1(\Omega)]{\mm_h^n}^2$ for all $\beta\in[0,1]$, with a hidden constant that depends only on $\alpha$ and $\tau$ (uniformly in $\beta$).

\emph{Step~2: Constraint violation.}
To get the unit length constraint error for all \(\beta\in[0,1]\), we
use~\cite[Proposition 3.4]{akrivis2025quadratic} or~\cite[Proposition 3.8]{afp2026_partI}.
Noting that nodewise \(|\mm_h^0|^2 = 1\) and
\(|\mm_h^1|^2 = |\mm_h^0|^2 + k^2|\vv_h^1|^2\),
we see that
\begin{equation}\label{eq:quadratic_inequality}
    \norm[L^1(\Omega)]{\II_h\big[|\mm_h^n|^2\big] - 1} \lesssim 
    k^2\norm{\vv_h^1}^2
    + k^4 \sum_{i=1}^{n-1}\norm{\dt^2\mm_h^{i+1}}^2,
\end{equation}
where the hidden constant depends only on the $\gamma$-quasi-uniformity of $\mathcal{T}_h$.

\emph{Case $\beta\in(0,1]$.} By~\eqref{eq:initial_energy_law} and~\eqref{eq:second_bound}, it follows that
\[
    \norm[L^1(\Omega)]{\II_h\big[|\mm_h^n|^2\big] - 1} \lesssim k^2\norm{\vv_h^1}^2 + k^2\Big(k^2\sum_{i=1}^{n-1}\norm{\dt^2\mm_h^{i+1}}^2\Big) \lesssim (1+\beta^{-1})\,k^2,
\]
which proves~\eqref{eq:unit_length_constraint}. Here, the hidden constant depends only on the $\gamma$-quasi-uniformity of $\mathcal{T}_h$, $\alpha$, and $\tau$, while the factor $\beta^{-1}$ stems from~\eqref{eq:second_bound}. In particular, $C_3$ depends on $\beta^{-1}$.

\emph{Third-order estimate, for all $\beta\in[0,1]$.}
Since \(\vv_h^1 = \vv_h^0 + k \dt\vv_h^1\),
we have from~\eqref{eq:quadratic_inequality} that
\begin{equation*}
	\norm[L^1(\Omega)]{\II_h\big[|\mm_h^n|^2\big] - 1} \lesssim 
	k^2\norm{\vv_h^0}^2
	+ k^4 \norm{\dt\vv_h^1}^2 + k^4 \sum_{j=1}^{n-1}\norm{\dt^2\mm_h^{j+1}}^2.
\end{equation*}
If \(\vv_h^0 \equiv \0\), we have
\begin{equation*}
	\norm[L^1(\Omega)]{\II_h\big[|\mm_h^n|^2\big] - 1} \lesssim 
	k^3 \Big(k \norm{\dt\vv_h^1}^2 + k\sum_{j=1}^{n-1}\norm{\dt^2\mm_h^{j+1}}^2\Big),
\end{equation*}
which is of third order in time if the quantity in brackets is uniformly bounded.
Under~\eqref{eq: regularity condition}, this proves~\eqref{eq:unit_length_constraint_third}, with $C_5$ that depends only on the $\gamma$-quasi-uniformity of $\mathcal{T}_h$ and on $C_4$.

\emph{Case $\beta=0$.} In this case, \cite[Proposition 3.4]{akrivis2025quadratic} or~\cite[Proposition 3.8]{afp2026_partI} still applies, but~\eqref{eq:second_bound} is not available, and~\eqref{eq:stability} provides no bound on $k^2\sum_{j=1}^{n-1}\norm{\dt^2\mm_h^{j+1}}^2$. 
Proceeding as in~\cite{afp2026_partI} and using the discrete inverse estimate in time from~\cite[Lemma~2.2]{akrivis2025quadratic}, i.e., $k^4\sum_{j=1}^{n-1}\norm{\dt^2\mm_h^{j+1}}^2\le\frac{18}{7}k^2\sum_{j=0}^{n-1}\norm{\vv_h^{j+1}}^2$, we obtain that
\begin{align*}
    \norm[L^1(\Omega)]{\II_h\big[|\mm_h^n|^2\big]-1}&\lesssim k^2\norm{\vv_h^1}^2+k^4 \sum_{j=1}^{n-1}\norm{\dt^2\mm_h^{j+1}}^2\lesssim k^2\norm{\boldsymbol{v}_h^1}^2+k^2 \sum_{j=0}^{n-1}\norm{\boldsymbol{v}_h^{j+1}}^2 \mathrel{\stackrel{\eqref{eq:initial_energy_law},\eqref{eq: bounds_intermediate}}{\lesssim}} k,
\end{align*}
where the hidden constant depends only on the $\gamma$-quasi-uniformity of $\mathcal{T}_h$,
$\alpha$ and $\tau$.
This shows the first-order constraint error, yielding~\eqref{eq:unit_length_constraint_bad}.
Note that this argument does not use $\beta=0$; hence, \eqref{eq:unit_length_constraint_bad} holds for all $\beta\in[0,1]$ with $C_2$ independent of $\beta$, although for $\beta\in(0,1]$ the estimate~\eqref{eq:unit_length_constraint} is sharper.

\emph{Step~3: $\bL^2$-bound and conclusion.} To bound the $\bL^2$-norm of $\mm_h^n$, we proceed as follows. Since $|\mm_h^n|^2$ is convex on each element $K\in\mathcal{T}_h$, it holds that $|\mm_h^n|^2\le\II_h\big[|\mm_h^n|^2\big]$ pointwise in $\Omega$. Together with~\eqref{eq:unit_length_constraint_bad}, which holds for all $\beta\in[0,1]$, this yields
\begin{align*}
    \norm{\mm_h^n}^2
    &\le\norm[L^1(\Omega)]{\II_h\big[|\mm_h^n|^2\big]}
    \le|\Omega|+\norm[L^1(\Omega)]{\II_h\big[|\mm_h^n|^2\big]-1}
    \le|\Omega|+C_2k\lesssim 1.
\end{align*}
Together with~\eqref{eq:stability} with $\norm{\Grad\mm_h^n}^2$ in place of $\norm[\bH^1(\Omega)]{\mm_h^n}^2$, proved in Step~1 for all $\beta\in[0,1]$, this proves~\eqref{eq:stability} and completes the proof.
\end{proof}

\subsection{Proof of Theorem~\ref{thm: main result}}

We are finally in a position to prove convergence
of the time reconstructions toward a weak solution of~\eqref{eq: iLLG system}.

\begin{proof}[Proof of Theorem~\ref{thm: main result}]
The proof is based on the energy method for parabolic PDEs, already employed in the seminal papers~\cite{alouges2008a,bp2006} to prove convergence of the first-order TPS and of an implicit midpoint-type scheme for the LLG equation, respectively.
That proof, however, requires modifications due to the
different time discretization (with ours based on the BDF2 method)
and to the presence of the inertial term, which are inspired by~\cite{afp2026_partI} and~\cite{ruggeri2022numerical}, respectively.
For the sake of clarity,
we split the proof into four steps, each of them verifying one of the four properties of Definition~\ref{def: weak solution}.

\emph{Step~1: Construction of $\mm$ satisfying Definition~\ref{def: weak solution}{\rm(\ref{def: weak solution i})}.}
The stability estimate~\eqref{eq:stability} from Proposition~\ref{prop:stability}
and the norm equivalence
\begin{equation*}
\frac{k}{5}\sum_{j=0}^{n-1}\norm{\vv_{h}^{j+1}}^2
\le
k\sum_{j=0}^{n-1}\norm{\dt\mm_{h}^{j+1}}^2
\le
\frac{9k}{7}\sum_{j=0}^{n-1}\norm{\vv_{h}^{j+1}}^2
\quad
\text{for all } n \ge 2
\end{equation*}
from~\cite[Lemma~2.1]{akrivis2025quadratic}
imply uniform boundedness of the time reconstructions for any choice of $\beta \in [0,1]$.
In particular, $\{\mm_{hk}\}$, $\{\mm_{hk}^\pm\}$, and $\{\widehat{\mm}_{hk}^+\}$ are uniformly bounded in $L^\infty(0,\infty;\bH^1(\Omega))$,
$\{\mm_{hk}\}$ is uniformly bounded also in $\bH^1(\OT)$ for all $T>0$,
while $\{\vv_{hk}^\pm\}$ is uniformly bounded in $L^\infty(0,\infty;\bL^2(\Omega))$.
Moreover, by~\eqref{eq:initial_energy_law} and~\eqref{eq:stability}, it holds that $\norm{\vv_h^i}\le C$ for all $i\ge1$, with $C>0$ independent of $h$, $k$, and $i$. Since $\dt\mm_h^1=\vv_h^1$ and, by~\eqref{eq: bound dt m j+1},
\[
    \norm{\dt\mm_h^{j+1}}
    \le\frac{2}{3}\norm{\vv_h^{j+1}}+\frac{1}{3}\norm{\dt\mm_h^{j}}
    \qquad\text{for all } j\ge1,
\]
an induction on $j$ yields $\norm{\dt\mm_h^{j+1}}\le C$ for all $j\ge0$. Consequently, $\{\partial_t\mm_{hk}\}$ is uniformly bounded in $L^\infty(0,\infty;\bL^2(\Omega))$, i.e., $\{\mm_{hk}\}$ is uniformly bounded in $W^{1,\infty}(0,\infty;\bL^2(\Omega))$.
Hence, with successive extractions of (not relabeled) weakly (resp., weakly-star) convergent subsequences
as guaranteed by the Eberlein--\v{S}mulian (resp., Banach--Alaoglu) theorem,
combined with standard results about continuous and compact embeddings of Sobolev spaces,
we obtain the existence of $\mm\in L^\infty(0,\infty;\bH^1(\Omega))\cap W^{1,\infty}(0,\infty;\bL^2(\Omega))$
such that
$\mm_{hk} \stackrel{*}{\rightharpoonup}\mm$ in $W^{1,\infty}(0,\infty;\bL^2(\Omega))$,
$\mm_{hk},\mm_{hk}^\pm,\widehat{\mm}_{hk}^+\stackrel{*}{\rightharpoonup}\mm$ in $L^\infty(0,\infty;\bH^1(\Omega))$, $\mm_{hk}^\pm,\mhat_{hk}^+\to \mm$ in $\bL^2(\OT)$, 
and $\vv_{hk}^\pm \stackrel{*}{\rightharpoonup} \vv \equiv \mmt$ in $L^{\infty}(0,\infty;\bL^2(\Omega))$,
as well as
$\mm_{hk}\rightharpoonup \mm$ in $\bH^1(\OT)$,
$\mm_{hk}\to \mm$ in $\bL^2(\OT)$,
and $\vv_{hk}^\pm\rightharpoonup \mmt$ in $\bL^2(\OT)$ for all $T>0$,
where all convergences are meant for $(h,k)\to(0,0)$;
see~\cite[Proposition~3.12]{afp2026_partI} and~\cite{ruggeri2022numerical} for all details, and, in particular, for the identification of the abstract limits.
Passing to the limit in~\eqref{eq:unit_length_constraint_bad} (for $\beta=0$) and in~\eqref{eq:unit_length_constraint} (for $\beta \in (0,1]$) as $(h,k)\to(0,0)$
shows that $\lvert \mm \rvert = 1$ a.e.\ in $\Omega \times (0,\infty)$.

\emph{Step~2: The limit $\mm$ satisfies Definition~\ref{def: weak solution}{\rm(\ref{def: weak solution iii})}.}
Let $\beta\in(0,1]$. The proof is split into five substeps, where the core argument is summarized in Substep~2a.

\emph{Substep~2a: Overview.}
Let $T>0$. Let $N \ge 1$ be such that $Nk = T$.
Let \(\vvarphi \in C^\infty_c([0,T); \CC^\infty(\overline{\Omega}))\)
be a smooth test function and let \(j=0,1,\ldots,N-1\).
We extend $\vvarphi$ by zero for $t\ge T$.
Let $\vvarphi_k$ and $\vvarphi_k^-$ denote the time
reconstructions~\eqref{eq:reconstructions} associated with
$\vvarphi(t_j)$.
By the smoothness of $\vvarphi$, we have
\[
    \vvarphi_k^- \to \vvarphi,
    \qquad
    \partial_t\vvarphi_k \to \partial_t\vvarphi
    \quad\text{in }L^\infty(0,T;\bW^{1,\infty}(\Omega))
    \quad\text{as }k\to0.
\]
We choose the test function
\[
    \zzeta^j_h = \boldsymbol{\II}_h[\mhat_{h}^{j+1}\times\vvarphi(t_j)]\in\bT_h(\mhat_h^{j+1}),
\]
where \(\mhat_{h}^{1} = \mm_h^0\) for uniform notation.
Testing the variational
formulations~\eqref{alg1:magnetization_update_first}--\eqref{alg1:magnetization_update}
with \(k\zzeta_h^j\),
we obtain for $j=0$
\begin{equation}\label{eq: first step discrete variational}
    \begin{split}
    \tau k&\inner{\dt\vv_h^1}{\boldsymbol{\II}_h[\mm_h^0 \times \vvarphi(t_0)]}
    + \alpha k \inner{\vv_h^1}{\boldsymbol{\II}_h[\mm_h^0 \times \vvarphi(t_0)]}\\
    &+ k\inner{\mm_h^0 \times \vv_h^1}{\boldsymbol{\II}_h[\mm_h^0 \times \vvarphi(t_0)]}
    = -k \inner{\Grad(\mm_h^0 + k \vv_h^1)}{\Grad\boldsymbol{\II}_h[\mm_h^0 \times \vvarphi(t_0)]}
    \end{split}
\end{equation}
and for \(j=1,\ldots,N-1\) 
\begin{align*}
   \beta \tau k&\inner{\dt\vv_h^{j+1}}{\boldsymbol{\II}_h[\mhat_h^{j+1} \times \vvarphi(t_j)]}
   + (1-\beta)\frac{\tau}{2}\inner{3\vv_h^{j+1}-4\vv_h^j+\vv_h^{j-1}}{\boldsymbol{\II}_h[\mhat_h^{j+1} \times \vvarphi(t_j)]}\\
    &+ \alpha k \inner{\vv_h^{j+1}}{\boldsymbol{\II}_h[\mhat_h^{j+1} \times \vvarphi(t_j)]}
    + k\inner{\mhat_h^{j+1} \times \vv_h^{j+1}}{\boldsymbol{\II}_h[\mhat_h^{j+1} \times \vvarphi(t_j)]}\\
   & = -\frac{k}{3} \inner{\Grad(4\mm_h^j - \mm_h^{j-1} + 2k\vv_h^{j+1})}{\Grad\boldsymbol{\II}_h[\mhat_h^{j+1} \times \vvarphi(t_j)]}.
\end{align*}
Summing these equations over \(j=0,\ldots,N-1\) yields that
\begin{align}
    \mkern-10mu\beta \tau k&\sum_{j=0}^{N-1}\inner{\dt\vv_h^{j+1}}{\boldsymbol{\II}_h[\mhat_h^{j+1} \times \vvarphi(t_j)]}
    + (1-\beta)\frac{\tau}{2}\sum_{j=1}^{N-1}\inner{3\vv_h^{j+1}-4\vv_h^j+\vv_h^{j-1}}{\boldsymbol{\II}_h[\mhat_h^{j+1} \times \vvarphi(t_j)]}\nonumber \\
    &+ (1-\beta) \tau k \inner{\dt\vv_h^1}{\boldsymbol{\II}_h[\mm_h^0\times \vvarphi(t_0)]}
    + \alpha k\sum_{j=0}^{N-1}\inner{\vv_h^{j+1}}{\boldsymbol{\II}_h[\mhat_h^{j+1} \times \vvarphi(t_j)]}\label{eq: discrete_variational_summation}\\
    &+ k\sum_{j=0}^{N-1}\inner{\mhat_h^{j+1} \times \vv_h^{j+1}}{\boldsymbol{\II}_h[\mhat_h^{j+1} \times \vvarphi(t_j)]}
    = -k\sum_{j=0}^{N-1}\inner{\Grad \mm_h^{j+1}}{\Grad\boldsymbol{\II}_h[\mhat_h^{j+1} \times \vvarphi(t_j)]}.\nonumber
\end{align}
In the following Substeps~2b--2e, we will show
that passing~\eqref{eq: discrete_variational_summation} to the limit $(h,k)\to(0,0)$
yields that
\begin{equation}\label{eq:limit_variational_summation}
    \begin{split}
    \int_0^T\inner{\partial_t\mm(t)}{\vvarphi(t)}\dtt
    = - \int_0^T\inner{\heff[\mm(t)]}{\vvarphi(t)\times \mm(t)}\dtt
    + \alpha \int_0^T\inner{\mm(t)\times\partial_t\mm(t)}{\vvarphi(t)}\dtt\\
    -\tau \int_0^T\inner{\mm(t)\times\partial_t\mm(t)}{\partial_t\vvarphi(t)}\dtt
    -\tau \inner{\mm^0\times\vv^0}{\vvarphi(0)}.
    \end{split}
\end{equation}
Overall, we will thus obtain that $\mm$ satisfies~\eqref{eq:limit_variational_summation} for all \(\vvarphi \in C^\infty_c([0,T); \CC^\infty(\overline{\Omega}))\),
and hence, by density,
for all \(\vvarphi \in C^\infty_c([0,T); \bH^1(\Omega))\).
This will prove the validity of Definition~\ref{def: weak solution}(\ref{def: weak solution iii}).

\emph{Substep~2b: Convergence of the three `LLG-terms' of~\eqref{eq: discrete_variational_summation}.}
The term on the right-hand side and the last two terms
on the left-hand side of~\eqref{eq: discrete_variational_summation}
appear also in the convergence proof of the BDF2-based algorithm for the standard LLG equation of~\cite{afp2026_partI}.
Since the convergence results fulfilled by the sequences of time reconstructions in that proof
are satisfied also here, these `LLG-terms' of~\eqref{eq: discrete_variational_summation}
behave in the very same way.
Specifically (see~\cite{afp2026_partI} for further details),
it holds that 
\begin{equation*}
	\begin{split}
		&\alpha k \sum_{j=0}^{N-1}\inner{\vv_h^{j+1}}{\boldsymbol{\II}_h[\widehat{\mm}_h^{j+1}\times\vvarphi(t_j)]}
        =\alpha \int_0^T \inner{\vv_{hk}^+(t)}{\boldsymbol{\II}_h[\widehat{\mm}_{h k}^+(t)\times\vvarphi_k^-(t)]}\dtt\\
		&\xrightarrow{(h,k)\to (0,0)}\alpha \int_0^T \inner{\partial_t\mm(t)}{\mm(t)\times\vvarphi(t)}\dtt
        = -\alpha\int_0^T\inner{\mm(t)\times\partial_t\mm(t)}{\vvarphi(t)}\dtt,
	\end{split}
\end{equation*}
\begin{equation*}
	\begin{split}
		&k\sum_{j=0}^{N-1} \inner{\widehat{\mm}_h^{j+1}\times\vv_h^{j+1}}{\boldsymbol{\II}_h[\widehat{\mm}_h^{j+1}\times\vvarphi(t_j)]}
        =\int_0^T\inner{\widehat{\mm}_{h k}^+(t)\times\vv_{hk}^+(t)}{\boldsymbol{\II}_h[\widehat{\mm}_{h k}^+(t)\times
        \vvarphi_k^-(t)]}\dtt\\
		&\mkern+26mu\xrightarrow{(h,k)\to(0,0)}\int_0^T\inner{\partial_t\mm(t)}{\vvarphi(t)}\dtt,
	\end{split}
\end{equation*}
and
\begin{equation*}
	\begin{split}
		&-k\sum_{j=0}^{N-1}\inner{\Grad\mm_h^{j+1}}{\Grad \boldsymbol{\II}_h[\widehat{\mm}_h^{j+1}\times\vvarphi(t_j)]}
        =-\int_0^T\inner{\Grad\mm_{hk}^+(t)}{\Grad \boldsymbol{\II}_h[\widehat{\mm}_{h k}^+(t)\times \vvarphi_k^-(t)]}\dtt\\
		&\mkern+18mu\xrightarrow{(h,k)\to(0,0)}-\int_0^T\inner{\Grad\mm(t)}{\mm(t)\times \Grad\vvarphi(t)}\dtt
        = -\int_0^T\inner{\heff[\mm(t)]}{\vvarphi(t)\times\mm(t)}\dtt.
	\end{split}
\end{equation*}
To conclude the proof of the variational formulation,
it thus remains to analyze the behavior in the limit of the first three terms on the left-hand side of~\eqref{eq: discrete_variational_summation},
which are those arising from the iLLG-inherent inertial term.

\emph{Substep~2c: Convergence of the first term of~\eqref{eq: discrete_variational_summation}.}
To deal with the first term, we proceed as in~\cite{ruggeri2022numerical}.
Using a summation by parts formula and some algebraic manipulations, \(\vvarphi(t_{N}) = \0\), and \(\mhat_h^1 = \mm_h^0\), we have
\begin{align*}
    \beta \tau &k\sum_{j=0}^{N-1}\inner{\dt\vv_h^{j+1}}{\boldsymbol{\II}_h[\mhat_h^{j+1} \times \vvarphi(t_j)]}
    =\beta \tau \sum_{j=0}^{N-1}\inner{\vv_h^{j+1}-\vv_h^j}{\boldsymbol{\II}_h[\mhat_h^{j+1} \times \vvarphi(t_j)]} \\  
    & = -\beta \tau \sum_{j=0}^{N-1}\inner{\vv_h^{j+1}}{\boldsymbol{\II}_h[\mhat_h^{j+2} \times \vvarphi(t_{j+1}) - \mhat_h^{j+1} \times \vvarphi(t_j)]} \\
    &\mkern+200mu+ \beta \tau \inner{\vv_h^N}{\boldsymbol{\II}_h[\mhat_{h}^{N+1}\times\vvarphi(t_{N})]} - \beta \tau \inner{\vv_h^0}{\boldsymbol{\II}_h[\mhat_{h}^{1}\times\vvarphi(t_{0})]}\\
    &= -\beta \tau \sum_{j=0}^{N-1}\inner{\vv_h^{j+1}}{\boldsymbol{\II}_h[\mhat_h^{j+2} \times \vvarphi(t_{j+1}) - \mhat_h^{j+1} \times \vvarphi(t_{j})]}- \beta \tau \inner{\vv_h^0}{\boldsymbol{\II}_h[\mm_h^0\times\vvarphi(t_{0})]}\\
    &= -\beta \tau k\sum_{j=0}^{N-1}\inner{\vv_h^{j+1}}{\boldsymbol{\II}_h[\mhat_h^{j+1} \times \dt\vvarphi(t_{j+1})]}
    -\beta \tau k\sum_{j=0}^{N-1}\inner{\vv_h^{j+1}}{\boldsymbol{\II}_h[\dt\mhat_h^{j+2} \times\vvarphi(t_{j+1})]}\\
    &\mkern+400mu-\beta \tau \inner{\vv_h^0}{\boldsymbol{\II}_h[\mm_{h}^{0}\times\vvarphi(t_{0})]}.
\end{align*}
As in the LLG case (see \cite{afp2026_partI}), standard interpolation estimates allow one to remove the nodal interpolants on the right-hand side at the cost of a remainder $o(1)$ as $(h,k)\to(0,0)$; see, e.g.,~\cite{afp2026_partI,alouges2008a}. Hence, using that $\aaa \cdot (\aaa \times \bb) = 0$, it holds that
\begin{equation}\label{eq:convergence_term_beta}
    \begin{split}
    \beta \tau k\sum_{j=0}^{N-1}\inner{\dt\vv_h^{j+1}}{\boldsymbol{\II}_h[\mhat_h^{j+1} \times \vvarphi(t_j)]}
    &= -\beta \tau k\sum_{j=0}^{N-1}\inner{\vv_h^{j+1}}{\mhat_h^{j+1} \times \dt\vvarphi(t_{j+1})}\\
    &\quad-\beta \tau k\sum_{j=0}^{N-1}\inner{\vv_h^{j+1}}{(\dt\mhat_h^{j+2} -\vv_h^{j+1})\times\vvarphi(t_{j+1})}\\
    & \quad-\beta \tau  \inner{\vv_h^0}{\mm_{h}^0\times\vvarphi(t_{0})}+o(1).
    \end{split}
\end{equation}
The first term of the right-hand side of~\eqref{eq:convergence_term_beta} converges to
\begin{equation*}
	\begin{split}
        -\beta \tau k& \sum_{j = 0}^{N-1}\inner{\vv_h^{j+1}}{\mhat_h^{j+1} \times \dt\vvarphi(t_{j+1})}
        = -\beta \tau \int_0^T\inner{\vv_{hk}^+(t)}{\widehat{\mm}_{hk}^+(t)\times \partial_t\vvarphi_k(t)}\dtt\\
        &\xrightarrow{(h,k)\to(0,0)}
        -\beta \tau \int_0^T\inner{\partial_t\mm(t)}{\mm(t)\times \partial_t\vvarphi(t)}\dtt
        = \beta \tau \int_{0}^{T}\inner{\mm(t)\times\partial_t\mm(t)}{\partial_t\vvarphi(t)}\dtt.
	\end{split}
\end{equation*}
Next, we deal with the second term of the right-hand side of~\eqref{eq:convergence_term_beta}.
For \(j=0\), we get \(\dt\mhat_h^{2}=2\vv_h^1\). Thus, it holds that
\[
    \inner{\vv_h^{1}}{(\dt\mhat_h^{2} -\vv_h^{1})\times\vvarphi(t_{1})}
        = \inner{\vv_h^{1}}{\vv_h^{1}\times\vvarphi(t_{1})}
        = 0,
\]
so that we can ignore the term with \(j=0\) in the second sum.
For $j\ge 1$, notice that 
\[
    \dt\mhat_h^{j+2} -\vv_h^{j+1} = \frac{k}{2}\dt^2\mm_h^{j+1}.
\]
Thus, we have that
\[
   - k\sum_{j=1}^{N-1}\inner{\vv_h^{j+1}}{(\dt\mhat_h^{j+2} -\vv_h^{j+1})\times\vvarphi(t_{j+1})}
   = - \frac{k^2}{2}\sum_{j=1}^{N-1}\inner{\vv_h^{j+1}}{\dt^2\mm_h^{j+1}\times\vvarphi(t_{j+1})}.
\]
Applying the weighted Young inequality
(with \(\varepsilon = k^{1/2}\)), it follows that
\[
    |\inner{\vv_h^{j+1}}{\dt^2\mm_h^{j+1}\times\vvarphi(t_{j+1})}|
    \leq \frac{\norm{\vv_h^{j+1}}^2}{2k^{1/2}}
    + k^{1/2}\frac{\norm{\dt^2\mm_h^{j+1}\times\vvarphi(t_{j+1})}^2}{2}
\]
and hence
\begin{align}\label{eq:term_to_0}
    &\Big| k\sum_{j=1}^{N-1}\inner{\vv_h^{j+1}}{(\dt\mhat_h^{j+2} -\vv_h^{j+1})\times\vvarphi(t_{j+1})}\Big|
    = \Big | \frac{k^2}{2}\sum_{j=1}^{N-1}\inner{\vv_h^{j+1}}{\dt^2\mm_h^{j+1}\times\vvarphi(t_{j+1})}\Big| \\
    &\leq \frac{k^{3/2}}{4}\sum_{j=1}^{N-1}\norm{\vv_h^{j+1}}^2
    \mkern-8mu+ \frac{k^{5/2}}{4}\sum_{j=1}^{N-1}\norm{\dt^2\mm_h^{j+1}\times\vvarphi(t_{j+1})}^2\!\stackrel{\eqref{eq:stability}}{\lesssim} \! k^{1/2}
    + k^{1/2}\norm[\bL^\infty(\OT)]{\vvarphi}^2 \xrightarrow{(h,k)\to(0,0)}0.  \nonumber  
\end{align}
Finally, the last term of the right-hand side of~\eqref{eq:convergence_term_beta} converges to
\[
    -\beta \tau \inner{\vv_h^0}{\mm_{h}^{0}\times\vvarphi(t_{0})}
    \xrightarrow{(h,k)\to(0,0)} -\beta \tau \inner{\vv^0}{\mm^0\times\vvarphi(0)}.
\]
Thus we have shown that
\begin{equation}\label{eq:convergence_first_term}
    \begin{split}
    \beta \tau k\sum_{j=0}^{N-1}\inner{\dt\vv_h^{j+1}}{\boldsymbol{\II}_h[\mhat_h^{j+1} \times \vvarphi(t_j)]}
    \xrightarrow{(h,k)\to(0,0)} &\beta \tau \int_0^T\inner{\mm(t)\times\partial_t\mm(t)}{\partial_t\vvarphi(t)}\dtt \\
    &- \beta \tau \inner{\vv^0}{\mm^0\times\vvarphi(0)}.
    \end{split}
\end{equation}
This concludes the proof of convergence for the
first term of~\eqref{eq: discrete_variational_summation}.

\emph{Substep~2d: Convergence of the second and third terms of~\eqref{eq: discrete_variational_summation}.}
To treat the second and third terms of~\eqref{eq: discrete_variational_summation},
we notice that, for $j\ge1$,
\[
    \frac{3}{2}\vv_h^{j+1}-2\vv_h^j+\frac{1}{2}\vv_h^{j-1}
    = k\dt\vv_h^{j+1}+\frac{k}{2}\big(\dt\vv_h^{j+1}-\dt\vv_h^{j}\big).
\]
Since $\mhat_h^1=\mm_h^0$, we can rewrite the second and third terms as
\begin{equation}\label{eq:second_term_split}
\begin{split}
    &(1-\beta)\tau\Big[\sum_{j=1}^{N-1}\Big\langle\frac{3}{2}\vv_h^{j+1}-2\vv_h^j+\frac{1}{2}\vv_h^{j-1},\boldsymbol{\II}_h[\mhat_h^{j+1}\times\vvarphi(t_j)]\Big\rangle
    + k\inner{\dt\vv_h^{1}}{\boldsymbol{\II}_h[\mm_h^{0}\times\vvarphi(t_0)]}\Big]\\
    =&(1-\beta)\tau k\Big[\sum_{j=0}^{N-1}\inner{\dt\vv_h^{j+1}}{\boldsymbol{\II}_h[\mhat_h^{j+1}\times\vvarphi(t_j)]}+\frac{1}{2}\sum_{j=1}^{N-1}\inner{\dt\vv_h^{j+1}-\dt\vv_h^{j}}{\boldsymbol{\II}_h[\mhat_h^{j+1}\times\vvarphi(t_j)]}\Big].
\end{split}
\end{equation}
Up to the factor $1-\beta$ instead of $\beta$, the first term on the right-hand side of~\eqref{eq:second_term_split} coincides with the term discussed in Substep~2c.
Hence, by~\eqref{eq:convergence_first_term}, it converges to
\[
    (1-\beta)\tau\int_0^T\inner{\mm(t)\times\partial_t\mm(t)}{\partial_t\vvarphi(t)}\dtt
    -(1-\beta)\tau\inner{\vv^0}{\mm^0\times\vvarphi(0)}.
\]
It remains to show that the second term on the right-hand side of~\eqref{eq:second_term_split} vanishes in the limit.
Using a summation by parts formula, \(\vvarphi(t_{N})=\0\), and \(\mhat_h^1=\mm_h^0\), we obtain
\begin{align}
    \frac{(1-\beta)\tau k}{2}&\sum_{j=1}^{N-1}\inner{\dt\vv_h^{j+1}-\dt\vv_h^{j}}{\boldsymbol{\II}_h[\mhat_h^{j+1}\times\vvarphi(t_j)]}\nonumber\\
    &=-\frac{(1-\beta)\tau k}{2}\sum_{j=0}^{N-1}\inner{\dt\vv_h^{j+1}}{\boldsymbol{\II}_h[\mhat_h^{j+2}\times\vvarphi(t_{j+1})-\mhat_h^{j+1}\times\vvarphi(t_j)]}\nonumber\\
    &\quad-\frac{(1-\beta)\tau k}{2}\inner{\dt\vv_h^{1}}{\boldsymbol{\II}_h[\mm_h^{0}\times\vvarphi(t_0)]}.\label{eq:sbp_second_term}
\end{align}
To deal with the first term on the right-hand side, note that
\[
    \mhat_h^{j+2}\times\vvarphi(t_{j+1})-\mhat_h^{j+1}\times\vvarphi(t_j)
    =k\,\dt\mhat_h^{j+2}\times\vvarphi(t_{j+1})+k\,\mhat_h^{j+1}\times\dt\vvarphi(t_{j+1}).
\]
Since, on shape-regular meshes, the $\bL^2$-norm of a discrete function is equivalent to the weighted $\ell^2$-norm of its nodal values, it follows that
\[
    \norm{\boldsymbol{\II}_h[\dt\mhat_h^{j+2}\times\vvarphi(t_{j+1})]}\lesssim\norm{\dt\mhat_h^{j+2}}
    \quad\text{and}\quad
    \norm{\boldsymbol{\II}_h[\mhat_h^{j+1}\times\dt\vvarphi(t_{j+1})]}\lesssim\norm{\mhat_h^{j+1}},
\]
where the hidden constants depend only on $\vvarphi$ and the shape regularity of $\T_h$.
Moreover, since $\mhat_h^1=\mm_h^0$ and $\mhat_h^{j+1}=2\mm_h^j-\mm_h^{j-1}$ for $j\ge1$, it holds that
\[
    \dt\mhat_h^{2}=\frac{(2\mm_h^1-\mm_h^0)-\mm_h^0}{k}=2\dt\mm_h^1
    \quad\text{and}\quad
    \dt\mhat_h^{j+2}=2\dt\mm_h^{j+1}-\dt\mm_h^{j}
    \quad\text{for } j\ge1.
\]
By Step~1, $\norm{\dt\mm_h^{j+1}}\lesssim1$ for all $j\ge0$, and hence $\norm{\dt\mhat_h^{j+2}}\lesssim1$ for all $j\ge0$.
Similarly, $\norm{\mm_h^j}\lesssim1$ for all $j\ge0$ by~\eqref{enum: A1} (for $j=0$),~\eqref{eq:initial_energy_law} together with $\mm_h^1=\mm_h^0+k\vv_h^1$ (for $j=1$), and~\eqref{eq:stability} (for $j\ge2$), which implies $\norm{\mhat_h^{j+1}}\lesssim1$ for all $j\ge0$.
This allows us to conclude that
\[
    k\sum_{j=0}^{N-1}\big(\norm{\dt\mhat_h^{j+2}}+\norm{\mhat_h^{j+1}}\big)^2
    \lesssim Nk = T.
\]
With the Cauchy--Schwarz inequality we thus obtain
\begin{align*}
    \frac{(1-\beta)\tau k}{2}&\Big|\sum_{j=0}^{N-1}\inner{\dt\vv_h^{j+1}}{\boldsymbol{\II}_h[\mhat_h^{j+2}\times\vvarphi(t_{j+1})-\mhat_h^{j+1}\times\vvarphi(t_j)]}\Big|\\
    &\lesssim k^2\sum_{j=0}^{N-1}\norm{\dt\vv_h^{j+1}}\big(\norm{\dt\mhat_h^{j+2}}+\norm{\mhat_h^{j+1}}\big)\\
    &\le k^{1/2}\Big(k^2\sum_{j=0}^{N-1}\norm{\dt\vv_h^{j+1}}^2\Big)^{1/2}\Big(k\sum_{j=0}^{N-1}\big(\norm{\dt\mhat_h^{j+2}}+\norm{\mhat_h^{j+1}}\big)^2\Big)^{1/2}\\
    &\mkern-5mu\stackrel{\eqref{eq:stability}}{\lesssim} k^{1/2}\xrightarrow{(h,k)\to(0,0)}0,
\end{align*}
where the hidden constant depends also on $\beta$, but not on $h$ and $k$.
To conclude, we still need to consider the last term on the right-hand side of the summation by parts~\eqref{eq:sbp_second_term} above, i.e., we must show that
\begin{equation}\label{eq:convergence_second_term_third}
    -\frac{(1-\beta) \tau k}{2}  \inner{\dt\vv_h^{1}}{\boldsymbol{\II}_h[\mm_h^{0} \times \vvarphi(t_{0})]}\xrightarrow{(h,k)\to(0,0)} 0.
\end{equation}
Indeed, recall the discrete variational formulation at the first time step~\eqref{eq: first step discrete variational}
    \begin{equation*}
        \tau k \inner{\dt\vv_h^1}{\zzeta_h^0}=
        - \alpha k \inner{\vv_h^1}{\zzeta_h^0}
        - k \inner{\mm_{h}^{0}\times \vv_h^1}{\zzeta_h^0}
        - k \inner{\Grad\mm_h^1}{\Grad \zzeta_h^0}.
    \end{equation*}
Moreover, recall that the energy estimate~\eqref{eq:initial_energy_law} guarantees that
$\norm{\vv_h^1} \lesssim1$ and $\norm{\Grad\mm_h^1} \lesssim 1$. Since 
$\mm_h^0$ assumes unit length at the nodes, we also have that $\norm[\bL^\infty(\Omega)]{\mm_h^0} \le 1$.
Finally, from the $\bH^1$-uniform boundedness of $\mm_h^0$ and the smoothness of $\vvarphi(t_0)$ it holds that $\norm[\bH^1(\Omega)]{\zzeta_h^0} = \norm[\bH^1(\Omega)]{\boldsymbol{\II}_h[\mm_h^0\times \vvarphi(t_0)]} \lesssim 1$. This implies that 
\begin{equation}\label{eq:first_step_test_bound}
    \begin{split}
        \tau k |\inner{\dt\vv_h^1}{\boldsymbol{\II}_h[\mm_h^0\times \vvarphi(t_0)]}|
        &\lesssim k \norm{\vv_h^1}\norm{\zzeta_h^0} + k \norm[\bL^\infty(\Omega)]{\mm_h^0}\norm{\vv_h^1}\norm{\zzeta_h^0} + k \norm{\Grad\mm_h^1}\norm{\Grad\zzeta_h^0}\\
        &\lesssim k \xrightarrow{(h,k)\to(0,0)} 0.
    \end{split}
\end{equation}
Combining~\eqref{eq:second_term_split}--\eqref{eq:first_step_test_bound}, we obtain that
\begin{equation}\label{eq:convergence_second_term}
\begin{split}
    (1-\beta)&\tau\sum_{j=1}^{N-1}\Big\langle\frac{3}{2}\vv_h^{j+1}-2\vv_h^j+\frac{1}{2}\vv_h^{j-1},\boldsymbol{\II}_h[\mhat_h^{j+1}\times\vvarphi(t_j)]\Big\rangle+(1-\beta)\tau k\inner{\dt\vv_h^{1}}{\boldsymbol{\II}_h[\mm_h^{0}\times\vvarphi(t_0)]}\\
    &\xrightarrow{(h,k)\to(0,0)}(1-\beta)\tau\int_0^T\inner{\mm(t)\times\partial_t\mm(t)}{\partial_t\vvarphi(t)}\dtt
    -(1-\beta)\tau\inner{\vv^0}{\mm^0\times\vvarphi(0)}.
\end{split}
\end{equation}

\emph{Substep~2e: Proof of~\eqref{eq:limit_variational_summation}.}
Combining the results of the previous substeps, in particular~\eqref{eq:convergence_first_term} and~\eqref{eq:convergence_second_term},
we obtain that
\begin{align*}
     &\mkern-10mu\beta \tau k\sum_{j=0}^{N-1}\inner{\dt\vv_h^{j+1}}{\boldsymbol{\II}_h[\mhat_h^{j+1} \times \vvarphi(t_j)]}
     + (1-\beta)\tau\sum_{j=1}^{N-1}\big\langle\frac{3}{2}\vv_h^{j+1}-2\vv_h^j+\frac{1}{2}\vv_h^{j-1},\boldsymbol{\II}_h[\mhat_h^{j+1} \times \vvarphi(t_j)]\big\rangle \nonumber\\
     &\mkern+380mu + (1-\beta) \tau k \inner{\dt\vv_h^{1}}{\boldsymbol{\II}_h[\mm_h^{0} \times \vvarphi(t_0)]}\nonumber\\
     &\mkern+55mu\xrightarrow{(h,k)\to(0,0)} \beta \tau \int_0^T\inner{\mm(t)\times\partial_t\mm(t)}{\partial_t\vvarphi(t)}\dtt
     - \beta \tau \inner{\vv^0}{\mm^0\times\vvarphi(0)}\label{eq:convergence_discrete_v_term}\\
     &\mkern+60mu+ (1-\beta) \tau \int_0^T\inner{\mm(t)\times\partial_t\mm(t)}{\partial_t\vvarphi(t)}\dtt
     - (1-\beta) \tau \inner{\vv^0}{\mm^0\times\vvarphi(0)}\nonumber\\
     &\mkern+60mu= \tau \int_0^T\inner{\mm(t)\times\partial_t\mm(t)}{\partial_t\vvarphi(t)}\dtt
     - \tau \inner{\vv^0}{\mm^0\times\vvarphi(0)}.\nonumber
\end{align*}
This shows that the discrete variational formulation~\eqref{eq: discrete_variational_summation}
converges to its continuous counterpart~\eqref{eq:limit_variational_summation} as \((h,k)\to(0,0)\).
As discussed in Substep~2a, this concludes the proof that the limit $\mm$ satisfies Definition~\ref{def: weak solution}(\ref{def: weak solution iii}).

\emph{Step~3: The limit $\mm$ satisfies Definition~\ref{def: weak solution}{\rm(\ref{def: weak solution iv})}.}
The proof is split into four substeps.

\emph{Substep~3a: Discrete energy inequality.}
Following the energy argument of~\cite{afp2026_PartIII}, we use the discrete energy law in the form~\eqref{eq: bounds_intermediate}, which was derived in the proof of Proposition~\ref{prop:stability}, with $\mathcal{J}_{G,h}^n$ and $\varrho_1$ defined therein.
Dropping the nonnegative terms on its left-hand side, we obtain
\[
    \mathcal{J}_{G,h}^n
    + \alpha k\sum_{j=0}^{n-1}\norm{\vv_h^{j+1}}^2
    \le \mathcal{J}[\mm_h^0,\vv_h^0]+\varrho_1
    \qquad\text{for }n\ge 2.
\]
In addition, notice that the same inequality holds for $n=1$, since
\eqref{eq:initial_energy_law} gives
\[
    \mathcal{J}_{G,h}^1+\alpha k\norm{\vv_h^1}^2
    =
    \mathcal{J}[\mm_h^0,\vv_h^0]+\varrho_1
    -\frac{\beta\tau k^2}{2}\norm{\dt\vv_h^1}^2.
\]
For $t'\in[0,T)$, let $n=n(t')$ be the unique index such that
$t'\in[t_{n-1},t_n)$.
Since $t'\le t_n$ and the dissipation is nonnegative,
we obtain
\begin{equation}\label{eq:discrete_energy_inequality}
    \mathcal{J}_{G,h}^{n(t')}
    +\alpha\int_0^{t'}\norm{\vv_{hk}^+(t)}^2\dtt
    \le \mathcal{J}[\mm_{hk}^-(0),\vv_{hk}^-(0)]+\varrho_1.
\end{equation}

\emph{Substep~3b: Passage to the limit.}
Let $I\subseteq[0,T]$ be measurable.
For $t'\in[t_{n-1},t_n)$, the pairs in
$\mathcal{J}_{G,h}^{n(t')}$ are given by
$(\mm_{hk}^+(t'),\mm_{hk}^-(t'))$ and
$(\vv_{hk}^+(t'),\vv_{hk}^-(t'))$; therefore, thanks to~\eqref{eq:G-norm}, it holds that
\begin{align*}
\mathcal{J}_{G,h}^{n(t')} &= \frac{1}{4}\big(\norm{\Grad \mm_{hk}^{+}(t')}^2 + \norm{2\Grad \mm_{hk}^{+}(t')-\Grad \mm_{hk}^{-}(t')}^2\big)\\
&\quad
+\frac{\beta \tau}{2}{\norm{\vv_{hk}^+(t')}^2}+\frac{1}{4}(1-\beta) \tau\big(\norm{\vv_{hk}^+(t')}^2+\norm{2\vv_{hk}^+(t')-\vv_{hk}^-(t')}^2\big).
\end{align*}
By Step~1,
\[
    \mm_{hk}^{\pm}\rightharpoonup\mm
    \quad\text{in }L^2(0,T;\bH^1(\Omega))
    \qquad\text{and}\qquad
    \vv_{hk}^{\pm}\rightharpoonup\mmt
    \quad\text{in }\bL^2(\OT).
\]
Hence, weak lower semicontinuity shows
\begin{equation}\label{eq:weak_lower_semicontinuity}
    \int_I\mathcal{J}[\mm(t'),\mmt(t')]\dtt'
    \le
    \liminf_{(h,k)\to(0,0)}
    \int_I\mathcal{J}_{G,h}^{n(t')}\dtt'.
\end{equation}
The strong convergence of the initial data from
assumptions~\eqref{enum: A1}--\eqref{enum: A2} gives
\begin{equation}\label{eq:convergence_initial_data}
   \mathcal{J}[\mm_{hk}^-(0),\vv_{hk}^-(0)]=\mathcal{J}[\mm_h^0,\vv_h^0]
    \longrightarrow
    \mathcal{J}[\mm^0,\vv^0].
\end{equation}
Finally, by convergence of $\vv_{hk}^+\rightharpoonup\partial_t\mm$ in $\bL^2(\OT)$,
weak lower semicontinuity on $(0,t')$
and Fatou's lemma yield
\begin{equation}\label{eq:convergence_signed_terms}
    \int_I\int_0^{t'}\norm{\mmt(t)}^2\dtt\,\dtt'
    \le
    \liminf_{(h,k)\to(0,0)}
    \int_I\int_0^{t'}\norm{\vv_{hk}^+(t)}^2\dtt\,\dtt'.
\end{equation}

\emph{Substep~3c: Control of the initial correction term.}
To deal with the term $\varrho_1$ on the right-hand side of~\eqref{eq:discrete_energy_inequality},
it is enough to prove that
\begin{equation}\label{eq: goal initial step}
\limsup_{(h,k)\to(0,0)}\varrho_1 \le 0.
\end{equation}
To this end, we insert the first-step energy identity~\eqref{eq:initial_energy_law} into the definition
of $\varrho_1$ to obtain
\begin{equation*}
    \begin{split}
        \varrho_1
        &= - \frac{\beta \tau}{4}
        (\norm{\vv_h^1}^2 - \norm{\vv_h^0}^2)
        -\frac{\tau}{4}\norm{\vv_h^1-\vv_h^0}^2
        - \frac{1}{4}\norm{\Grad \mm_h^1 - \Grad \mm_h^0}^2
        - \frac{\alpha k}{2} \norm{\vv_h^1}^2\\
        &\le \frac{\beta \tau}{4}
        (\norm{\vv_h^0}^2 - \norm{\vv_h^1}^2).
    \end{split}
\end{equation*}
Since $\vv_h^0\to\vv^0$ strongly in $\bL^2(\Omega)$,
it suffices to prove that
$
\liminf_{(h,k)\to(0,0)}\norm{\vv_h^1}^2 \ge \norm{\vv^0}^2.
$
Indeed, if this holds, then
\[
    \limsup_{(h,k)\to(0,0)}\varrho_1
    \le \frac{\beta \tau}{4}
    \big(\norm{\vv^0}^2
    - \liminf_{(h,k)\to(0,0)} \norm{\vv_h^1}^2\big)
    \le \frac{\beta \tau}{4}
    \big(\norm{\vv^0}^2 - \norm{\vv^0}^2\big)
    = 0,
\]
and \eqref{eq: goal initial step} follows.
To do so, choose a subsequence along which $\norm{\vv_h^1}^2$ converges to $\liminf_{(h,k)\to(0,0)}\norm{\vv_h^1}^2$.
Since $\vv_h^1$ is uniformly bounded in $\bL^2(\Omega)$ by~\eqref{eq:initial_energy_law}, we may extract a further subsequence (not relabeled) such that $\vv_h^1\rightharpoonup\overline{\vv}$ in $\bL^2(\Omega)$ for some $\overline{\vv}\in\bL^2(\Omega)$.
To obtain the required lower bound, we consider again
the discrete variational formulation at the first time
step~\eqref{eq: first step discrete variational}
\begin{equation*}
    \tau k \inner{\dt\vv_h^1}{\zzeta_h^0}
    =
    \tau \inner{\vv_h^1 - \vv_h^0}{\zzeta_h^0}
    =
    - \alpha k \inner{\vv_h^1}{\zzeta_h^0}
    - k \inner{\mm_h^0\times \vv_h^1}{\zzeta_h^0}
    - k \inner{\Grad\mm_h^1}{\Grad\zzeta_h^0}.
\end{equation*}
Passing to the limit as in~\eqref{eq:first_step_test_bound}, thanks to the strong
convergence of
$\zzeta_h^0 = \boldsymbol{\II}_h[\mm_h^0\times \vvarphi(t_0)]$
to $\mm^0\times \vvarphi(0)$ in $\bL^2(\Omega)$,
see \cite{afp2026_partI} for details, we obtain
\[
    \inner{\overline{\vv}-\vv^0}{\mm^0\times \vvarphi(0)}
    = 0
    \quad \text{for all }
    \vvarphi\in C^\infty_c([0,T);\CC^\infty(\overline{\Omega})).
\]
Since $\vvarphi(0)$ is arbitrary, this yields $\mm^0\times(\overline{\vv}-\vv^0)=\0$ a.e.\ in $\Omega$.
Taking the cross product with $\mm^0$ and using the vector triple-product identity $\aaa\times(\bb\times\boldsymbol{c})=(\aaa\cdot\boldsymbol{c})\bb-(\aaa\cdot\bb)\boldsymbol{c}$ together with $|\mm^0|=1$, we get
\[
    \0=\mm^0\times\big(\mm^0\times(\overline{\vv}-\vv^0)\big)
    =\big(\mm^0\cdot(\overline{\vv}-\vv^0)\big)\mm^0-(\overline{\vv}-\vv^0),
\]
i.e., $\overline{\vv}-\vv^0=\big(\mm^0\cdot(\overline{\vv}-\vv^0)\big)\mm^0$.
Since $\mm^0\cdot\vv^0=0$, this implies $(\overline{\vv}-\vv^0)\cdot\vv^0=0$.
Consequently,
\[
    \norm{\overline{\vv}}^2
    = \norm{\overline{\vv}-\vv^0}^2 + \norm{\vv^0}^2
    \ge \norm{\vv^0}^2.
\]
By the choice of the subsequence and weak lower semicontinuity,
we obtain
\[
    \liminf_{(h,k)\to(0,0)}\norm{\vv_h^1}^2
    \ge \norm{\overline{\vv}}^2
    \ge \norm{\vv^0}^2.
\]
This concludes the proof of~\eqref{eq: goal initial step}.

\emph{Substep~3d: Continuous energy inequality.}
Let $I\subseteq[0,T]$ be measurable.
Integrating~\eqref{eq:discrete_energy_inequality} over $I$
and using~\eqref{eq:weak_lower_semicontinuity}--\eqref{eq: goal initial step}, we obtain
\begin{align*}
    &\int_I\Big[
        \mathcal{J}[\mm(t'),\mmt(t')]
        +\alpha\int_0^{t'}\norm{\mmt(t)}^2\dtt
    \Big]\dtt'
    \\
    &\quad\le
    \liminf_{(h,k)\to(0,0)}
    \int_I\Big[
        \mathcal{J}_{G,h}^{n(t')}
        +\alpha\int_0^{t'}\norm{\vv_{hk}^+(t)}^2\dtt
    \Big]\dtt'
    \\
    &\quad\le
    \limsup_{(h,k)\to(0,0)}
    \int_I\Big[
        \mathcal{J}[\mm_h^0,\vv_h^0]+\varrho_1
    \Big]\dtt'
    \le
    \int_I\mathcal{J}[\mm^0,\vv^0]\dtt'.
\end{align*}
Since $I$ is arbitrary, this proves
\[
    \mathcal{J}[\mm(t'),\mmt(t')]
    +\alpha\int_0^{t'}\norm{\mmt(t)}^2\dtt
    \le \mathcal{J}[\mm^0,\vv^0]
    \qquad\text{for a.e.\ }t'\in[0,T].
\]
This shows that the limit $\mm$ satisfies the energy inequality from Definition~\ref{def: weak solution}(\ref{def: weak solution iv}) for almost all $t'\in[0,T]$.

To extend this to all times $t\in[0,T]$, we need pointwise-in-time values of $\mmt$. To this end, we first show that $\mmt$ admits a representative that is weakly continuous in time with values in $\bL^2(\Omega)$.

Let $\ppsi\in\bH^1(\Omega)$ and $\eta\in C^\infty_c(0,T)$.
Testing~\eqref{eq:weak_solution var_formulation} with $\vvarphi=\eta\ppsi$ (so that $\vvarphi(0)=\0$, $\partial_t\vvarphi=\eta'\ppsi$, and $\Grad\vvarphi=\eta\Grad\ppsi$) and recalling that $\inner{\heff[\mm]}{\vvarphi\times\mm} = \inner{\Grad\mm}{\mm\times\Grad\vvarphi}$, we obtain
\begin{equation*}
    \begin{split}
    \tau\int_0^T\eta'(t)\,\inner{\mm(t)\times\mmt(t)}{\ppsi}\dtt
    &= -\int_0^T\eta(t)\Big[\inner{\mmt(t)}{\ppsi}
    + \inner{\Grad\mm(t)}{\mm(t)\times\Grad\ppsi}\\
    &\qquad\qquad- \alpha\inner{\mm(t)\times\mmt(t)}{\ppsi}\Big]\dtt.
    \end{split}
\end{equation*}
Since $\eta\in C^\infty_c(0,T)$ is arbitrary, this states, by definition of the weak derivative, that the real-valued function $t\mapsto\inner{\mm(t)\times\mmt(t)}{\ppsi}$ is weakly differentiable with
\begin{equation}\label{eq:weak_derivative}
    \tau\frac{\mathrm{d}}{\mathrm{d}t}\inner{\mm(t)\times\mmt(t)}{\ppsi}
    = \inner{\mmt(t)}{\ppsi}
    + \inner{\Grad\mm(t)}{\mm(t)\times\Grad\ppsi}
    - \alpha\inner{\mm(t)\times\mmt(t)}{\ppsi}
\end{equation}
for almost every $t\in(0,T)$.
We now bound the right-hand side of~\eqref{eq:weak_derivative}.

By Step~1, $\mm\in L^\infty(0,T;\bH^1(\Omega))\cap W^{1,\infty}(0,T;\bL^2(\Omega))$; in particular, for a.e.\ $t\in(0,T)$,
\begin{equation*}
    \norm{\Grad\mm(t)}\le\norm[L^\infty(0,T;\bH^1(\Omega))]{\mm}
    \quad\text{and}\quad
    \norm{\mmt(t)}\le\norm[L^\infty(0,T;\bL^2(\Omega))]{\mmt}.
\end{equation*}
Moreover, since $|\mm|=1$ a.e.\ in $\OT$, the Cauchy--Schwarz inequality yields
\begin{align*}
    \tau\Big|\frac{\mathrm{d}}{\mathrm{d}t}\inner{\mm(t)\times\mmt(t)}{\ppsi}\Big|
    &\le (1+\alpha)\norm{\mmt(t)}\norm{\ppsi} + \norm{\Grad\mm(t)}\norm{\Grad\ppsi}\\
    &\le C\norm[\bH^1(\Omega)]{\ppsi},
\end{align*}
where $C\coloneqq(1+\alpha)\norm[L^\infty(0,T;\bL^2(\Omega))]{\mmt}+\norm[L^\infty(0,T;\bH^1(\Omega))]{\mm}$ is independent of $t$ and $\ppsi$.
This shows that the weak derivative of $t\mapsto\inner{\mm(t)\times\mmt(t)}{\ppsi}$ is bounded by $C\norm[\bH^1(\Omega)]{\ppsi}/\tau$ for a.e.\ $t\in(0,T)$.
By the fundamental theorem of calculus for weakly differentiable functions of one variable, this implies that, up to modification on a null set, $t\mapsto\inner{\mm(t)\times\mmt(t)}{\ppsi}$ is Lipschitz continuous. Since this bound is uniform with respect to $\ppsi\in\bH^1(\Omega)$, we conclude that
\[
    \mm\times\mmt\in W^{1,\infty}(0,T;\bH^1(\Omega)^*);
\]
in particular, up to modification on a null set, $\mm\times\mmt$ is Lipschitz continuous, and hence weakly continuous, in time with values in $\bH^1(\Omega)^*$.
Since, moreover, $\mm\times\mmt\in L^\infty(0,T;\bL^2(\Omega))$ and $\bL^2(\Omega)\hookrightarrow\bH^1(\Omega)^*$ continuously, Lemma~\ref{lem:temam}, applied with $X=\bL^2(\Omega)$ and $Y=\bH^1(\Omega)^*$, shows that $\mm\times\mmt$ is weakly continuous in time with values in $\bL^2(\Omega)$.

It remains to transfer this property from $\mm\times\mmt$ to $\mmt$.
Since $\mm\cdot\mmt=0$ and $|\mm|=1$ a.e.\ in $\OT$, by the vector triple-product identity
\begin{equation*}
    \mmt = -\mm\times(\mm\times\mmt)
    \quad\text{a.e.\ in } \OT.
\end{equation*}
Let $\ppsi\in\bL^\infty(\Omega)$.
Then, the identity $(\boldsymbol{a}\times\boldsymbol{b})\cdot\boldsymbol{c}=\boldsymbol{b}\cdot(\boldsymbol{c}\times\boldsymbol{a})$ gives
\begin{equation*}
    \inner{\mm\times(\mm\times\mmt)}{\ppsi} = \inner{\mm\times\mmt}{\ppsi\times\mm}.
\end{equation*}
Since $\mm\in W^{1,\infty}(0,T;\bL^2(\Omega))$ is Lipschitz continuous in time with values in $\bL^2(\Omega)$, so is $\ppsi\times\mm$: for all $s,t\in[0,T]$, it holds that
\begin{align*}
\norm{\ppsi\times\mm(t)-\ppsi\times\mm(s)}
&\le\norm[\bL^\infty(\Omega)]{\ppsi}\norm{\mm(t)-\mm(s)}\\
&\le\norm[\bL^\infty(\Omega)]{\ppsi}\norm[L^\infty(0,T;\bL^2(\Omega))]{\mmt}\,|t-s|.
\end{align*}
In particular, $\ppsi\times\mm$ is strongly continuous in time with values in $\bL^2(\Omega)$.
Since the inner product of a weakly convergent and a strongly convergent sequence converges, the map $t\mapsto\inner{\mm\times(\mm\times\mmt)}{\ppsi}$ is continuous.
By density of $\bL^\infty(\Omega)$ in $\bL^2(\Omega)$, and since all functions involved are bounded in $\bL^2(\Omega)$ uniformly in time, the same holds for all $\ppsi\in\bL^2(\Omega)$.
Hence, $-\mm\times(\mm\times\mmt)$ is a representative of $\mmt$ that is weakly continuous in time with values in $\bL^2(\Omega)$.
From now on, $\mmt$ denotes this representative.

To finally extend the energy inequality from almost all $t'\in[0,T]$ to all $t\in[0,T]$, we collect the continuity properties of the terms involved.
First, $\Grad\mm$ is weakly continuous in time with values in $\bL^2(\Omega)$.
Indeed, $\mm$ is Lipschitz continuous in time with values in $\bL^2(\Omega)$ and bounded in $L^\infty(0,T;\bH^1(\Omega))$, so that Lemma~\ref{lem:temam} with $X=\bH^1(\Omega)$ and $Y=\bL^2(\Omega)$ shows that $\mm$ is weakly continuous in time with values in $\bH^1(\Omega)$.
Second, $\mmt$ is weakly continuous in time with values in $\bL^2(\Omega)$ by construction.
Now, let $t\in[0,T]$.
Since the energy inequality holds for almost all $t'\in[0,T]$, we can choose a sequence of such times $t'_\ell$ with $t'_\ell\to t$ as $\ell\to\infty$.
By the first two properties, $\Grad\mm(t'_\ell)\rightharpoonup\Grad\mm(t)$ and $\mmt(t'_\ell)\rightharpoonup\mmt(t)$ weakly in $\bL^2(\Omega)$.
Since norms are weakly lower semicontinuous, this yields
\begin{equation*}
    \mathcal{J}[\mm(t),\mmt(t)]\le\liminf_{\ell\to\infty}\mathcal{J}[\mm(t'_\ell),\mmt(t'_\ell)].
\end{equation*}
Moreover, the dissipation integral is Lipschitz continuous in time, since
\begin{equation*}
    \Big|\int_s^t\norm{\mmt(r)}^2\,\mathrm{d}r\Big|
    \le|t-s|\,\norm[L^\infty(0,T;\bL^2(\Omega))]{\mmt}^2
    \qquad\text{for all } s,t\in[0,T],
\end{equation*}
so that $\int_0^{t'_\ell}\norm{\mmt(s)}^2\,\mathrm{d}s\to\int_0^{t}\norm{\mmt(s)}^2\,\mathrm{d}s$ as $\ell\to\infty$.
Altogether, we obtain
\begin{equation}\label{eq:energy_inequality_2}
\begin{aligned}
    &\mathcal{J}[\mm(t),\mmt(t)]
    +\alpha\int_0^t\norm{\mmt(s)}^2\,\mathrm{d}s\\
    &\quad\le\liminf_{\ell\to\infty}
    \Big[
        \mathcal{J}[\mm(t'_\ell),\mmt(t'_\ell)]
        +\alpha\int_0^{t'_\ell}\norm{\mmt(s)}^2\,\mathrm{d}s
    \Big]
    \le\mathcal{J}[\mm^0,\vv^0],
\end{aligned}
\end{equation}
where the last inequality is the energy inequality at the times $t'_\ell$.
Since $t\in[0,T]$ and $T>0$ are arbitrary, this proves Definition~\ref{def: weak solution}{\rm(\ref{def: weak solution iv})}.

\emph{Step~4: The limit $\mm$ satisfies Definition~\ref{def: weak solution}{\rm(\ref{def: weak solution ii})}.}
Finally, we aim to verify that the initial conditions
$(\mm^0,\vv^0)$ are satisfied continuously in
$\bH^1(\Omega)\times\bL^2(\Omega)$ as $t\to0$.

By $\mm_{hk}\rightharpoonup\mm$ in $\bH^1(\OT)$ in Step~1, and the continuity of the time-trace operator
from $\bH^1(\OT)$ to $\bL^2(\Omega)$,
\[
    \mm_h^0=\mm_{hk}(0)\rightharpoonup\mm(0)
    \qquad\text{in }\bL^2(\Omega).
\]
On the other hand, assumption~\eqref{enum: A1} gives
$\mm_h^0\to\mm^0$ strongly in $\bH^1(\Omega)$,
and hence also in $\bL^2(\Omega)$.
By uniqueness of the weak limit, we conclude that
$\mm(0)=\mm^0$.

Together with the Lipschitz continuity of $\mm$ and the weak continuity of $\Grad\mm$ in time with values in $\bL^2(\Omega)$, both established in Substep~3d,
this implies that
\[
    \mm(t)\to\mm^0
    \quad\text{and}\quad
    \Grad\mm(t)\rightharpoonup\Grad\mm^0
    \qquad\text{in }\bL^2(\Omega)
    \quad\text{as }t\to0.
\]
It remains to identify $\mmt(0)$.
Recall from Substep~3d that $\mmt=-\mm\times(\mm\times\mmt)$, where $\mm\times\mmt$ is weakly continuous in time with values in $\bL^2(\Omega)$ and, for all $\ppsi\in\bH^1(\Omega)$, the real-valued function $t\mapsto\inner{(\mm\times\mmt)(t)}{\ppsi}$ is Lipschitz continuous on $[0,T]$ with derivative given by~\eqref{eq:weak_derivative}.
Since $\mm(0)=\mm^0$, it thus suffices to identify $(\mm\times\mmt)(0)$.
To this end, let $\ppsi\in\bH^1(\Omega)$ and $\eta\in C^\infty_c([0,T))$ with $\eta(0)=1$, and choose $\vvarphi=\eta\ppsi$ in~\eqref{eq:weak_solution var_formulation}.
Integrating by parts in time in the term involving $\partial_t\vvarphi=\eta'\ppsi$ in~\eqref{eq:weak_solution var_formulation} and using~\eqref{eq:weak_derivative}, we obtain
\begin{align*}
    -\tau\int_0^T&\eta'(t)\inner{(\mm\times\mmt)(t)}{\ppsi}\dtt
    \\
    &=\tau\inner{(\mm\times\mmt)(0)}{\ppsi}
    +\tau\int_0^T\eta(t)\,\frac{\mathrm{d}}{\mathrm{d}t}\inner{(\mm\times\mmt)(t)}{\ppsi}\dtt
    \\
    &\mkern-5mu\stackrel{\eqref{eq:weak_derivative}}{=}\tau\inner{(\mm\times\mmt)(0)}{\ppsi}
    +\int_0^T\eta(t)\inner{\mmt(t)}{\ppsi}\dtt
    \\
    &\quad+\int_0^T\eta(t)\inner{\Grad\mm(t)}{\mm(t)\times\Grad\ppsi}\dtt
    -\alpha\int_0^T\eta(t)\inner{\mm(t)\times\mmt(t)}{\ppsi}\dtt.
\end{align*}
We insert this into~\eqref{eq:weak_solution var_formulation} and recall that $\inner{\heff[\mm]}{\vvarphi\times\mm}=\eta\inner{\Grad\mm}{\mm\times\Grad\ppsi}$.
Then all integrals cancel, and only the terms at $t=0$ remain, i.e.,
\[
    0=\tau\inner{(\mm\times\mmt)(0)}{\ppsi}-\tau\inner{\mm^0\times\vv^0}{\ppsi}.
\]
Since $\tau>0$, $\ppsi\in\bH^1(\Omega)$ is arbitrary, and $\bH^1(\Omega)$ is dense in $\bL^2(\Omega)$, this yields $(\mm\times\mmt)(0)=\mm^0\times\vv^0$ a.e.\ in $\Omega$.
Since $|\mm^0|=1$ and $\mm^0\cdot\vv^0=0$, the vector triple-product identity
$\aaa\times(\bb\times \boldsymbol{c}) = (\aaa\cdot \boldsymbol{c})\bb - (\aaa\cdot \bb)\boldsymbol{c}$
gives
\[
    \mmt(0)=-\mm^0\times(\mm\times\mmt)(0)=-\mm^0\times\big(\mm^0\times\vv^0\big)=\vv^0\quad \text{ a.e.\ in } \Omega.
\]
By the weak continuity of $\mmt$, we conclude that $\mmt(t)\rightharpoonup\vv^0$ in $\bL^2(\Omega)$ as $t\to0$.

To upgrade this to strong convergence, we expand the squared norms
and use~\eqref{eq:energy_inequality_2} to obtain
\begin{align*}
    0\le{}&
    \norm{\Grad\mm(t)-\Grad\mm^0}^2
    +\tau\norm{\mmt(t)-\vv^0}^2
    \\
    \le{}&
    2\norm{\Grad\mm^0}^2+2\tau\norm{\vv^0}^2
    -2\inner{\Grad\mm(t)}{\Grad\mm^0}
    -2\tau\inner{\mmt(t)}{\vv^0}.
\end{align*}
The right-hand side tends to zero as $t\to0$
by the weak convergences established above.
Since $\tau>0$, this proves
$\Grad\mm(t)\to\Grad\mm^0$ and
$\mmt(t)\to\vv^0$ strongly in $\bL^2(\Omega)$.
Together with $\mm(t)\to\mm^0$ in $\bL^2(\Omega)$,
this shows that the initial conditions are satisfied
continuously in $\bH^1(\Omega)\times\bL^2(\Omega)$.
This concludes the proof of Theorem~\ref{thm: main result}.
\end{proof}

\appendix
\section{}\label{sec:appendix}
\begin{lemma}[{\cite[Chapter~III, Lemma~1.4]{temam1977navier}}]\label{lem:temam}
Let $X$ and $Y$ be Banach spaces with $X\subset Y$ and continuous injection, and let $X$ be reflexive.
Let $T>0$ and let $u\in L^\infty(0,T;X)$ be weakly continuous in time with values in $Y$, i.e., for all $g\in Y^*$, the map $t\mapsto\dual[Y^*\times Y]{g}{u(t)}$ is continuous on $[0,T]$.
Then, $u(t)\in X$ with $\norm[X]{u(t)}\le\norm[L^\infty(0,T;X)]{u}$ for all $t\in[0,T]$, and $u$ is weakly continuous in time with values in $X$, i.e., for all $f\in X^*$, the map $t\mapsto\dual[X^*\times X]{f}{u(t)}$ is continuous on $[0,T]$.
\end{lemma}
\section*{Acknowledgements}
The first author acknowledges support from the \emph{Vienna School of Mathematics}.
The first and third authors are supported
by the Austrian Science Fund (FWF) through projects
10.55776/F65,
10.55776/PAT3699424,
and 10.55776/PAT3446525.
The fourth author is a member of the `Gruppo Nazionale per il Calcolo Scientifico (GNCS)'
of the Italian `Istituto Nazionale di Alta Matematica (INdAM)',
and
was partially supported by GNCS (CUP E53C25002010001)
and by the European Cooperation in Science and Technology 
(COST Action POLYTOPO CA23134).
The authors thank Shuo Yang (Beijing Institute of Mathematical Sciences and Applications)
for useful discussions.

\printbibliography


\end{document}